\documentclass[12pt]{article}
\usepackage{amssymb}
\usepackage{amsmath}
\usepackage{amsthm}
\usepackage{graphicx}
\usepackage{graphics}
\usepackage{diagbox}
\usepackage{wrapfig}
\usepackage{tikz}
\usepackage{pgf,tikz}
\usetikzlibrary{fit}
\usepackage{xcolor,hyperref}

\usepackage{cancel} 
\usetikzlibrary{patterns} 

\newcommand{\R}{\mathbb{R}} 
\newcommand{\exponente}{\vartheta}

\theoremstyle{theorem}
\newtheorem{thmintro}{Theorem}

\newtheorem{thmBp}{Theorem}

\newtheorem{thm}{Theorem}[section]
\newtheorem{prop}[thm]{Proposition}
\newtheorem{lem}[thm]{Lemma}
\newtheorem{cor}[thm]{Corollary}
\newtheorem{qs}[thm]{Question}
\theoremstyle{definition}
\newtheorem{defn}[thm]{Definition}
\newtheorem{rem}[thm]{Remark}
\newtheorem{ex}[thm]{Example}

\newcommand{\ger}{\widehat{\mathrm{Diff}}(\mathbb{K},0)}
\newcommand{\geri}{\widehat{\mathrm{Diff}_{\ell}}(\mathbb{K},0)}
\newcommand{\geriplus}{\widehat{\mathrm{Diff}}_{\ell + 1}(\mathbb{K},0)}

\input epsf

\begin{document}

\date{}
\author{H\'el\`ene Eynard-Bontemps \,\, \& \,\, Andr\'es Navas}

\title{On residues and conjugacies for germs of 1-D parabolic diffeomorphisms in finite regularity}
\maketitle

\noindent{\bf Abstract.} We study conjugacy classes of germs of non-flat diffeomorphisms of the real line fixing the origin. Based on the work of Takens and Yoccoz, we establish results that are sharp in terms of differentiability classes and order of tangency to the identity. The core of all of this lies in the invariance of residues under low-regular conjugacies. This may be seen as an extension of the fact (also proved in this article) that the value of the Schwarzian derivative at the origin for germs of $C^3$ parabolic diffeomorphisms is invariant under $C^2$ parabolic conjugacy, though it may vary arbitrarily under parabolic $C^1$ conjugacy.

\vspace{0.2cm}


\vspace{0.2cm}

\noindent{\bf MSC-class:} 20F38, 37C25, 37E99. 

\vspace{0.5cm}

Several dynamical properties of maps or flows strongly depend on the behavior at fixed points or singularities of vector fields, respectively. Normal forms then become a crucial tool for their study. There is a huge literature in this regard for complex analytic maps. In particular, a complete classification in the 1-D case follows from the work of \'Ecalle \cite{ecalle}, Martinet and Ramis \cite{MR}, and Voronin \cite{voronin}. The real 1-D case was treated by Takens \cite{takens} in the $C^{\infty}$ setting and later studied by Yoccoz \cite{yoccoz} in finite regularity. It follows from their works that for a non-flat germ of diffeomorphism,  the associated vector field ({\em i.e.} the unique $C^1$ vector field whose time-1 map yields the germ), first constructed by Szekeres \cite{szekeres}, has the expected regularity ($C^{\infty}$ in case of a germ of smooth maps, $C^{k-1}$ for a germ of $C^k$ diffeomorphisms). This certainly has direct consequences for the conjugacy classes of the corresponding germs of diffeomorphisms, yet these are explicit only in \cite{takens}, with no mention in \cite{yoccoz}. This work arose from the necessity of clarifying this issue. More interestingly, some of our results go beyond a direct translation of those of Takens and Yoccoz,  and reveal several unexpected rigidity phenomena for low-regular conjugacies.

Most of this work will deal with germs of diffeomorphisms of the real line fixing the origin. To simplify, we will restrict the discussion to orientation-preserving (germs of) maps; extending the results/examples to the orientation-reversing case is straightforward. Moreover, we will slightly modify the notion of germ by considering only right neighborhoods of the origin. In other words, a {\em germ} in this setting will be an equivalence class of diffeomorphisms of intervals of the form $[0,\eta)$, with $\eta > 0$, where two maps are equivalent if they coincide on some interval  $[0,\eta')$, with $\eta' > 0$. Definitively, this definition is more adapted to the 1-D context; moreover, one can easily retrieve analogous results for the original notion starting from those in this setting.

In the framework above, an obvious conjugacy invariant is the value of the derivative at the fixed point. In the hyperbolic case, this value is a complete invariant of $C^{1+\tau}$ conjugacy. (The $C^1$ case is different and rather special; see \cite{eynard-navas-sternberg}.) For non-flat parabolic germs, the first invariant is the order of contact to the identity at the origin.  Although the sign of the corresponding coefficient in the Taylor series expansion is invariant under conjugacy by germs of orientation-preserving maps, its value is not, as it can be easily modified by conjugacies by homotheties. However, there is a number that yields another obstruction to conjugacy: the {\em residue} $\mathrm{Res}$. Together with its companion, the {\em iterative residue} $\mathrm{Resit}$,  this is a very well-known object in the complex setting, and extends to non-flat germs of  regular-enough (for instance, real-analytic) diffeomorphisms as follows: Start with a parabolic germ of the form 
 $$f(x) = x + \sum_{n \geq \ell+1} a_n \, x^n, \qquad \mbox{with } \, \, a_{\ell+1} \neq 0.$$
 Such a germ will be said to be {\em exactly $\ell$-tangent to the identity}. Moreover, $f$ will be said to be 
 (topologically) 
 {\em expanding} (resp. {\em contracting}) if the value of the coefficient $a_{\ell+1}$ above is positive (resp. negative). This is equivalent to that $f(x) > x$ (resp. $f(x) < x$) for all small-enough $x > 0$. A very well-known lemma claims that $f$ is conjugate by a germ of diffeomorphism to a germ of the (reduced) form 
$$x \mapsto x \pm x^{\ell+1} + \mu x^{2\ell + 1} + \sum_{n \geq 2\ell + 2} b_n \, x^n,$$
where the sign above is $+$ (resp. $-$) in the expanding (resp. contracting) case. Moreover, the value of $\mu$ is uniquely determined. One then defines 
$$\mathrm{Res} (f) := \mu \qquad \mbox{ and } \qquad \mathrm{Resit} (f) := \frac{\ell + 1}{2} - \mu.$$
By definition, these expressions are invariant under conjugacy. 

In \S \ref{section-morphisms}, we put this construction in a more algebraic perspective. In particular, we consider formal germs of diffeomorphisms at the origin (or changes of variables) on arbitrary ground fields. In this general setting, in \S \ref{subsection:resadd}, we introduce still another residue, that we call {\em additive residue}. For a germ $f$ as above, this is just defined as 
$$\mathrm{Resad}_{\ell} (f) := \frac{(\ell+1) \, a_{\ell+1}^2}{2} - a_{2\ell+1}.$$
This residue fails to be invariant under conjugacy; however, it has the (somewhat surprising and) very nice property of being a group homomorphism when restricted to the subgroup of parabolic germs having order of contact to the identity at least $\ell$. Using this, in \S \ref{section-resitad}, we give a very short and conceptual proof of the important  iterative formula for $\mathrm{Resit}$ first established by Écalle: for $n \neq 0$,
$$\mathrm{Resit}(f^n) = \frac{\mathrm{Resit}(f)}{|n|}.$$
An interesting relation with the classical Schwarzian derivative is presented in \S \ref{section-sc}.

\vspace{0.1cm}

The definition of residues easily extends to germs of $C^{2\ell+1}$ diffeomorphisms that are exactly $\ell$-tangent to the identity (see \S \ref{section-resitad}). It is not hard to see that $\mathrm{Res}$ and $\mathrm{Resit}$ are invariant under conjugacy by $C^{2\ell+1}$ diffeomorphisms. The first two main results of this work, stated and discussed in~\S \ref{section-statements}, establish that invariance actually holds for conjugacies by germs of $C^{\ell+1}$ diffeomorphisms, and that this result is sharp. 

\begin{thmintro}
\label{t:A}
Given $\ell \geq 1$, let $f,g$ be two parabolic germs in $\mathrm{Diff}^{2\ell + 1}_+(\mathbb{R},0)$ that are exactly $\ell$-tangent to the identity. If $f$ and $g$ are conjugated by a germ in $\mathrm{Diff}^{\ell + 1}_+ (\mathbb{R},0)$, then they have the same (iterative) residue.
\end{thmintro}

\begin{thmintro}
\label{t:B}
Given $\ell \geq 1$, let $f,g$ be two parabolic germs in $\mathrm{Diff}^{2\ell + 1}_+(\mathbb{R},0)$ 
that are exactly $\ell$-tangent to the identity. If they are both expanding or both contracting,  
then they are conjugated by a germ in $\mathrm{Diff}^{\, \ell}_+ (\mathbb{R},0)$.
\end{thmintro}

Roughly, these results summarize (and are nicely complemented) as follows:

\vspace{0.1cm}

\noindent -- Two non-flat parabolic germs of (real 1-D) diffeomorphisms are conjugate by the germ of (orientation-preserving) $C^1$ diffeomorphisms if and only if they have the same order of contact to the identity and the first (higher-order) nonzero coefficients of the Taylor series expansions have the same sign (this goes back to Firmo \cite{firmo});

\vspace{0.1cm}

\noindent -- Any $C^1$ conjugacy is automatically of class $C^{\ell}$ if the order of  tangency above is $\ell$ and the germs are of class $C^{2\ell + 1}$;

\vspace{0.1cm}

\noindent -- An obstruction to improve the differentiability of the conjugacy above is given by the iterative residue: for $C^{2\ell + 1}$ diffeomorphisms that are exactly $\ell$-tangent to the identity, this value is invariant under $C^{\ell + 1}$ conjugacy.

\vspace{0.1cm}

After discussing some relevant examples and heuristic arguments in \S \ref{section-statements}, a complete proof of Theorem \ref{t:A} (resp. Theorem \ref{t:B}) is presented in \S \ref{section-proofs} (resp. \S \ref{section-parte2}). Actually, Theorem \ref{t:B} is proved in the following extended form:

\begin{thmBp} 
\label{t:Bp} 
Let $1 \leq r \leq \ell$, and let $f,g$ be two parabolic germs in $\mathrm{Diff}^{\ell + 1 + r}_+(\mathbb{R},0)$ that are exactly $\ell$-tangent to the identity. If they are both expanding or both contracting, then they are conjugated by a germ in $\mathrm{Diff}^{\, r}_+ (\mathbb{R},0)$.
\end{thmBp}

To close this work, we study the regularity of the conjugacies in case of coincidence of residues. For germs of $C^{\infty}$ diffeomorphisms with the same order of tangency to the identity, Takens proved  that they are $C^{\infty}$ conjugate provided they have the same iterative residue. Our last main result is a version of this phenomenon  in finite regularity. Our proof strongly follows his arguments (and allows retrieving his result without using Borel's lemma; see Remark \ref{r:final}).

\begin{thmintro}  
\label{t:C}
Let $r \geq \ell + 1$ and let $f,g$ be germs of $C^{\ell+1+r}$ diffeomorphisms, both expanding or both contracting, that are exactly $\ell$-tangent to the identity, with $\ell \geq 1$. If $f$ and $g$ have the same (iterative) residue, then they are $C^{\ell+1}$ conjugate, and such a conjugacy is automatically 
a (germ of) $C^r$ diffeomorphism. 
\end{thmintro}

In particular (take $r = \ell + 1$), two germs of $C^{2\ell + 2}$ diffeomorphisms with the same iterative residue that are both expanding or both contracting are $C^{\ell+1}$ conjugate. Notice that this fails to be true for germs of $C^{2\ell+1}$ diffeomorphisms, as is shown by an example in \S \ref{section-ejemplo-helene}.

\vspace{0.2cm}

This article is written in a pedagogic style because some of the issues discussed here are folklore. In particular, several sections presenting examples and alternative arguments may be skipped in a first reading by those who want to directly proceed to the core of the proofs. We hope that our presentation will help the readers in getting a panorama on the subject, and will perhaps invite them to work on it. A couple of concrete questions on distortion elements of the group of germs of diffeomorphisms is addressed in the last section.


\section{Some general (algebraic) facts}
\label{section-morphisms}

\subsection{Germs and morphisms}
\label{subsection:resadd}

Let $\mathbb{K}$ be an integral domain. We let $\ger$ be the group of formal power series of the form
$$a_1 x + \sum_{n\geq 2} a_n x^n,$$
where $a_n \in \mathbb{K}$ and $a_1 \in \mathbb{K} \setminus \{0\}$ is invertible.  The group operation is that of composition of series, which identifies to that of substitution of variables. In the case where $\mathbb{K}$ is a field, it also identifies to the group of formal germs of diffeomorphisms at the origin. 

For $\ell \geq 1$, we consider the subgroup $\geri$ of $\ger$ formed by the germs that are $\ell$-tangent to the identity, that is, that are of the form 
$$x + \sum_{n \geq \ell + 1} a_n x^n.$$ 
On $\ger$, there is an obvious homomorphism $\Phi_0$ into the multiplicative subgroup of the invertible elements of $\mathbb{K} \setminus \{0\}$ given by 
$$\Phi_0 : \sum_{n\geq 1} a_n x^n \mapsto a_1.$$
Similarly, on each subgroup $\geri$, there is an obvious homomorphism $\Phi_{\ell}$ into the additive group $\mathbb{K}$, namely, 
$$\Phi_\ell : x + \sum_{n \geq \ell + 1} a_n x^n \mapsto a_{\ell + 1} .$$ 
Actually, $\Phi_{\ell}$ is the first of $\ell$ homomorphisms $\Phi_{\ell,1}, \Phi_{\ell,2} \ldots, \Phi_{\ell,\ell}$ defined on $\geri$ by
$$\Phi_{\ell,i}: x + \sum_{n \geq \ell + 1} a_n x^n \mapsto a_{\ell+i}$$
(checking that these are homomorphisms is straightforward; see also (\ref{sale}) below).

Less trivially, there is another homomorphism from $\geri$ into $\mathbb{K}$, namely 
\begin{equation}\label{def-resad}
\overline{\mathrm{Resad}}_{\ell} \!: x + \sum_{n \geq \ell + 1} a_n x^n \mapsto  (\ell + 1) \, a_{\ell + 1}^2 - 2 a_{2\ell + 1}.
\end{equation}
Whenever it makes sense (for instance, when $\mathbb{K}$ is a zero-characteristic field), we let
\begin{equation}\label{eq-resad-i}
\mathrm{Resad}_{\ell} \!: x + \sum_{n \geq \ell + 1} a_n x^n \mapsto  \frac{\ell + 1}{2} \, a_{\ell + 1}^2 - a_{2\ell + 1}.
\end{equation}
We call both \, $\overline{\mathrm{Resad}}_{\ell}$ \, and \, $\mathrm{Resad}_{\ell}$ \, the {\em additive residues} at order $\ell + 1$. 
Checking that they yield group homomorphisms is straightforward: if 
$$f (x) = x + \sum_{n \geq \ell + 1} a_n x^n \quad \mbox{ and } \quad g (x) = x + \sum_{n \geq \ell + 1} a_n' x^n,$$ 
then $fg(x)$ equals 
$$ \left[ x + \sum_{m \geq \ell + 1} a_m' x^m \right] +  \sum_{n \geq \ell + 1} a_n  \left[ x + \sum_{m \geq \ell + 1} a_m' x^m \right]^n,$$
that is,
$$ x + \sum_{m =  \ell + 1}^{2\ell + 1} a_m' x^m + \sum_{n =  \ell + 1}^{2\ell + 1} a_n x^n  + a_{\ell + 1}  (\ell + 1) x^{\ell} a_{\ell + 1}' x^{\ell + 1}  + T_{2\ell+2},$$
where $T_{2\ell+2}$ stands for a formal power series in $x$ all of whose terms have order at least $2\ell+2$. Therefore, 
\begin{equation}\label{sale}
fg (x) = x + \sum_{n=\ell + 1}^{2 \ell} [a_n+a_n'] x^n + [a_{2\ell + 1} + a'_{2\ell + 1} + (\ell + 1) a_{\ell + 1} a_{\ell + 1}'] x^{2\ell + 1} + T_{2\ell+2},
\end{equation}
and 
$$\overline{\mathrm{Resad}}_{\ell} (fg) 
= (\ell + 1) \, [a_{\ell + 1} + a_{\ell + 1}']^2  - 2 \,  [a_{2\ell + 1} + a'_{2\ell + 1} + (\ell + 1) a_{\ell + 1} a_{\ell + 1}'] = 
\overline{\mathrm{Resad}}_{\ell} (f) + \overline{\mathrm{Resad}}_{\ell} (g). 
$$

\noindent{\bf A basis for the cohomology.} In the case where $\mathbb{K}$ is the field of real numbers, Fukui proved in \cite{fukui} that the $\Phi_{\ell,i}$ together with $\mathrm{Resad}_{\ell}$ are the generators of the continuous cohomology group of $\geri$. In \cite{babenko}, Babenko and Bogatyi treated the case $\mathbb{K} = \mathbb{Z}$. In particular, they computed the continuous cohomology $H^1 (\widehat{\mathrm{Diff}}_1 (\mathbb{Z},0))$: besides 
$$\Phi_{1}:  x + \sum_{n \geq 2} a_n x^n \mapsto a_2$$ 
and 
$$\mathrm{Resad}_1: x + \sum_{n \geq 2} a_n x^n \mapsto a_2^2 - a_3,$$ 
there are two homomorphisms into $\mathbb{Z}_2$, namely
$$x + \sum_{n \geq 2} a_n x^n \mapsto \frac{a_2 \, ( 1 + a_2 )}{2} + a_3 + a_4 \quad (\mbox{mod } 2)$$
and
$$x + \sum_{n \geq 2} a_n x^n \mapsto a_2  a_4 + a_4 + a_6 \quad (\mbox{mod } 2).$$
In \cite{BB}, Bogatayaa and Bogatyi do similar computations for $\mathbb{K} = \mathbb{Z}_p$, with $p$ a prime number. 
They show that, for $p \neq 2$, only  $\Phi_{1}$ and $\mathrm{Resad}_1$ survive, but for $p=2$, the situation is slightly more complicated. 
See also \cite{pavez}, where among other things 
it is shown that $H^1 ( \widehat{\mathrm{Diff}}_{\ell}(\mathbb{Z},0))$ is finitely generated for all values of $\ell$ larger than $1$. 

\vspace{0.4cm}

\noindent{\bf Genuine diffeomorphisms.} For $\mathbb{K}$ being the fields of the real, complex or $p$-adic numbers, we let $\mathrm{Diff}^{\omega} (\mathbb{K},0)$ be the group of genuine germs of analytic diffeomorphisms fixing the origin, and $\mathrm{Diff}^{\omega}_{\ell} (\mathbb{K},0)$ the subgroup of those elements that are $\ell$-tangent to the identity. These are subgroups of the corresponding groups $\ger$ and $\geri$, so the homomorphisms $\Phi_{\ell,i}$ and $\mathrm{Resad}_{\ell}$ restrict to them. 

More interestingly, 
each of the homomorphisms $\Phi_{\ell,i}$ and $\mathrm{Resad}_{\ell}$ above involve only finitely many derivatives at the origin (namely, $\ell + i$ and $2\ell + 1$, respectively). Their definitions hence extend to the subgroup $\mathrm{Diff}^{\infty}_{\ell} (\mathbb{K},0)$ of the group $\mathrm{Diff}^{\infty} (\mathbb{K},0)$ of germs of $C^{\infty}$ diffeomorphisms fixing the origin made of the elements that are $\ell$-tangent to the identity. Actually, they even extend to the larger group of germs of $C^{2\ell + 1}$ diffeomorphisms that are $\ell$-tangent to the identity. In this framework,  it is a nice exercise to check the additive property of $\mathrm{Resad}_{\ell}$ just by using the Fa\`a di Bruno formula. We will come back to this point in \S \ref{section-sc}.

Fukui's theorem stated above still holds in $\mathrm{Diff}^{\infty}_+ (\mathbb{R},0)$  (where $+$ stands for orientation-preserving maps). 
We do not know whether a complex or $p$-adic version of this is also valid. 


\subsection{Res-it-ad}
\label{section-resitad}

Elements in $\geri \setminus \geriplus$ are those whose order of contact to the identity at the origin equals $\ell$. 
For simplicity, they will be said to be {\em exactly $\ell$-tangent to the identity}. These correspond to series expansions of the form 
$$x + \sum_{n \geq \ell + 1} a_n x^n, \quad \mbox{ with } \,\, a_{\ell + 1} \neq 0.$$
A well-known lemma gives a normal form for these elements.

\begin{lem}  \label{lem:raro}
If $\mathbb{K}$ is a zero-characteristic field having a square root of $-1$, then every $f~\in~\geri \setminus \geriplus$ is conjugate  in $\widehat{\mathrm{Diff}} (\mathbb{K},0)$ to a (unique) germ of the (normal) form
\begin{equation} \label{most-general-normal-form}
x \mapsto x + x^{\ell + 1} + \mu x^{2\ell + 1}.
\end{equation}
\end{lem}
 
An explicit argument showing this in the complex case that applies with no modification to the present case appears in \cite{voronin}. The value of $\mu$ is called the {\em residue} of $f$, and denoted $\mathrm{Res} (f)$. As it is very well known (and easy to check), this is invariant under conjugacy. Actually, if $f$ with normal form (\ref{most-general-normal-form}) is conjugate in $\ger$ to a germ of the (reduced) form
$$x + x^{\ell + 1} + \sum_{n \geq 2\ell + 1} a_n x^{n},$$
then necessarily $a_{2\ell + 1} = \mathrm{Res} (f)$. 

The residue was introduced as a conjugacy invariant in the complex setting: for a germ of a holomorphic map fixing the origin, one defines 
$$\mathrm{R}(f) := \frac{1}{2 \pi i} \int_{\gamma} \frac{dz}{z - f(z)},$$
where $\gamma$ is a small, simple and positively oriented loop around the origin. For germs of the form 
$$f (z) = z +  z^{\ell + 1} +  \sum_{n \geq  2\ell + 1} a_n z^n,$$
one easily checks that $\mathrm{R}(f) = \mathrm{Res} (f) = a_{2\ell+1}$. 

The {\em iterative residue} of $f \in \geri \setminus \geriplus$, denoted $\mathrm{Resit} (f)$, is simply the value 
$$\frac{\ell + 1}{2} - \mathrm{Res} (f).$$
This satisfies the fundamental relation 
\begin{equation}\label{eq:ecalle-resit}
\mathrm{Resit} (f^n) = \frac{\mathrm{Resit}(f)}{|n|}
\end{equation}
for every integer $n \neq 0$. It was introduced by \'Ecalle in the context of germs of holomorphic diffeomorphisms \cite{ecalle}. 
A  proof of the formula above in this setting (due to \'Ecalle) appears for instance in \cite{milnor}. Below  we give an algebraic 
proof that works in a much broader context (in particular, it covers the case of germs of diffeomorphisms of the real line with 
finite regularity discussed just after Lemma \ref{lem-basic}).  
As we will see, it just follows from the additive properties of $\Phi_{\ell}$ and $\mathrm{Resad}_{\ell}$. 

Start with a germ of the form
$$f(x) = x + x^{\ell + 1} + \mu x^{2\ell + 1} +  \sum_{n \geq 2\ell+2} a_n x^n.$$
Using the homomorphisms $\Phi_{\ell}$ and $\mathrm{Resad}_{\ell}$, we obtain 
$$f^n (x) = x + n x^{\ell + 1} + \mu_n x^{2\ell + 1} + T_{2\ell+2},$$
with 
$$\frac{(\ell + 1) \, n^2}{2} - \mu_n = \mathrm{Resad}_{\ell} (f^n) = n \, \mathrm{Resad}_{\ell} (f) = n \, \left[ \frac{(\ell + 1)}{2} - \mu \right] = n \, \mathrm{Resit}(f).$$
Assume $n > 0$. If we conjugate by $H_{\lambda_n} \! :  x \mapsto \lambda_n x$ with $\lambda_n := \sqrt[\ell]{n}$, we obtain 
$$H_{\lambda_n} f^n H_{\lambda_n}^{-1} (x) = x + x^{\ell + 1} + \frac{\mu_n}{n^2} x^{2\ell + 1} + T_{2\ell+2}'.$$
By definition, this implies that 
$$\mathrm{Resit} (f^n) =  \frac{\ell + 1}{2} - \frac{\mu_n}{n^2} 
= \frac{1}{n^2} \left[ \frac{(\ell + 1) \, n^2}{2} - \mu_n \right] = 
\frac{ n \, \mathrm{Resit} (f) }{n^2} = \frac{\mathrm{Resit}(f)}{n}.$$
as announced. Computations for $n < 0$ are analogous.

\begin{ex} It follows from (the proof of) Lemma \ref{lem:raro} (see also Lemma \ref{lem-basic} below) 
that, for $f \in \geri \setminus \geriplus$ of the form
$$f(x) = x + \sum_{n \geq \ell + 1} a_n \, x^n, \qquad a_{\ell + 1} \neq 0,$$
the value of $\mathrm{Resit} (f)$ is a polynomial function of $a_{\ell + 1},a_{\ell+2}, \ldots, a_{2\ell + 1}$ times some integer power of $a_{\ell + 1}$. 
For instance, the reader may readily check that, for $a \neq 0$, 
$$\mathrm{Resit} (x + ax^2+bx^3+ \ldots) = \frac{a^2 - b}{a^2},$$
$$\mathrm{Resit} (x+ax^3+bx^4+cx^5+\ldots ) 
= \frac{3a^3 - b^2 - ac}{2a^3}.$$
See \cite[Theorem 1]{juan} for a general result in this direction. 
\end{ex}

\noindent  {\bf The residues in the real case.} Much of the discussion above still makes sense 
for germs of $C^{2\ell + 1}$ diffeomorphisms. In particular, the morphisms $\Phi_{\ell,i}$ 
and $\mathrm{Resad}_{\ell}$ extend to the subgroup $\mathrm{Diff}^{2\ell + 1}_{\ell} (\mathbb{R},0)$ made of those germs that are $\ell$-tangent to the identity. Moreover, the definition of $\mathrm{Resit}$ still extends to the elements of this subgroup that are exactly $\ell$-tangent to the identity at the origin, because of the next very well known lemma. (Compare Lemma \ref{lem:raro}.) For the statement, recall that a germ $f$ of diffeomorphisms of the real line fixing the origin is (topologically) 
{\em contracting} (resp. {\em expanding}) if  $f(x) < x$ (resp. $f(x)>x$) holds for all small-enough $x >0$.  
In all what follows, the little-$o$-notation $o (\cdot)$ will refer to values of the involved variable 
that are positive but very close to zero. 

\begin{lem} \label{lem-basic} Let $f$ be the germ of a $C^{2\ell + 1}$ diffeomorphism of the real line 
that is exactly $\ell$-tangent to the identity at the origin. If $f$ is expanding (resp. contracting), then there exists a germ of polynomial diffeomorphism $h$ such that the conjugate germ $hfh^{-1}$ has 
a Taylor series expansion at the origin of the (reduced) form
$$hfh^{-1} (x) = x \pm x^{\ell + 1} + \mu x^{2\ell + 1} + o (x^{2\ell + 1}),$$
where the sign is positive (resp. negative) in the expanding (resp. contracting) case. Moreover, $\mu$ is uniquely determined up to conjugacy by a germ of $C^{2\ell + 1}$ diffeomorphism.
\end{lem}

\begin{proof} By assumption, $f$ has a Taylor series expansion at the origin of the form
$$f(x) = x + \sum_{n=\ell + 1}^{2\ell + 1} a_n x^n + o (x^{2\ell + 1}), \quad \mbox{with } a_{\ell + 1} \neq 0.$$
By conjugating $f$ by a homothety, we can arrange that $a_{\ell + 1} = 1$ (resp.  $a_{\ell + 1} = -1$) in the expanding (resp. contracting) case. Now, by conjugating by a germ of the form $x+\alpha x^2$, we can make $a_{\ell+2}$ vanish without changing $a_{\ell + 1} = \pm1$. Next we conjugate 
by a germ of the form $x+\alpha x^3$ to make $a_{\ell+3}$ vanish without changing $a_{\ell + 1} = \pm 1$ and $a_{\ell+2}=0$. One continues this way up to conjugating by a germ of the form $x+\alpha x^{\ell}$, which allows killing $a_{2\ell}$. We leave it to the reader to fill in the details as well as the proof of the uniqueness of $\mu$ (see also \cite{voronin}).
\end{proof}

In any case above, we let 
$$\mathrm{Res} (f) :=  \mu 
\qquad \mbox{ and } \qquad 
\mathrm{Resit} (f) = \frac{\ell + 1}{2} - \mu.$$
One readily checks that this definition is coherent with that of the complex analytic case. 

\begin{rem}
\label{rem:diferenciable}
Strictly speaking, for a germ of diffeomorphism that is $\ell$-tangent to the identity, one does not need $C^{2\ell+1}$ regularity to define residues: $(2\ell + 1)$-differentiability at the origin is enough. Indeed, all definitions only use Taylor series expansions at the origin and their algebraic properties under composition.
\end{rem}


\subsection{Schwarzian derivatives and $\mathrm{Resad}_{\ell}$}
\label{section-sc}

 For a germ of a genuine diffeomorphism $f \in \mathrm{Diff}^{\infty}(\mathbb{K},0)$ (where $\mathbb{K}$ is, for example, the field of real numbers), one has 
$$a_n = \frac{D^n f (0)}{n!}.$$
Thus, if $Df (0) = 1$, then 
$$\mathrm{Resad}_1 (f) = a_2^2 - a_3 = \frac{(D^2 f (0))^2}{4} - \frac{D^3f(0)}{6} = 
\frac{1}{6} \left[ \frac{3}{2} \left( \frac{D^2 f(0)}{Df(0)} \right)^2 - \frac{D^3f (0)}{Df(0)} \right] = \frac{Sf (0)}{6},$$
where $Sf$ denotes the Schwarzian derivative. That $Sf$ appears as a group homomorphism is not surprising. Indeed, in a general setting, it satisfies the cocycle identity 
\begin{equation}\label{cocicle-s}
S (fg) = S(g) + S (f) \circ g \cdot Dg^2,
\end{equation}
and restricted to parabolic germs this becomes 
$$S(fg)(0) = Sg (0) + Sf (0).$$
In this view, $\mathrm{Resad}_{\ell}$ may be seen as a kind of ``Schwarzian derivative of higher order'' for parabolic germs.  
Now, for $f \in \mathrm{Diff}_{\ell}^{\infty} (\mathbb{K},0)$, we have
$$\mathrm{Resad}_{\ell} (f) 
= \frac{\ell + 1}{2} \left( \frac{D^i f (0)}{(\ell + 1)\,!} \right)^2 - \frac{D^{2\ell + 1}f(0)}{(2\ell + 1)\,!} 
= \frac{1}{(2\ell + 1)\, !} \left[ {2\ell + 1 \choose \ell} \frac{(D^{\ell + 1} f(0))^2}{2} - D^{2\ell + 1} f (0) \right] \! .$$
Thus, letting
$$S_{\ell + 1} (f) := {2\ell + 1 \choose \ell} \frac{(D^{\ell + 1} f(0))^2}{2} - D^{2\ell + 1} f (0) = (2\ell+1)! \, \mathrm{Resad}_{\ell}(f),$$
we have $S_1 (f) = S(f)$. 

In a slightly broader context, if $f,g$ are two germs of $C^{2\ell + 1}$ diffeomorphisms  with all derivatives $D^2, D^3, \ldots ,D^{\ell}$ vanishing at the origin then, using the Fa\`a di Bruno formula, it is straightforward to check the cocycle relation below: 
$$S_{\ell} (fg) = S_{\ell} (g) + S_{\ell} (f)  \cdot (Dg )^{2\ell}.$$
This is an extension of the additive property of $\mathrm{Resad}_{\ell}$ to the framework of non-parabolic germs.

\vspace{0.1cm}

\begin{rem} 
A classical theorem of Szekeres establishes that every expanding or contracting germ of $C^2$ diffeomorphism of the real line is the time-1 map of a flow of  germs of $C^1$ diffeomorphisms;  moreover, a famous lemma of Kopell establishes that its centralizer in the group of germs of $C^1$ diffeomorphisms coincides with this flow (see \cite[chapter 4]{book}). Furthermore, it follows from the work of Yoccoz \cite{yoccoz}  that, for non-flat germs of class $C^k$, the flow is made of $C^{k-1}$ diffeomorphisms. Therefore, to each $f \in \mathrm{Diff}^{2\ell + 2}_{\ell} (\mathbb{R},0) \setminus \mathrm{Diff}^{2\ell + 2}_{\ell+1} (\mathbb{R},0)$  we may associate its flow $(f^t)$ made of $C^{2\ell+1}$ germs of diffeomorphisms. These are easily seen to be exactly $\ell$-flat. Moreover, equality (\ref{eq:ecalle-resit}) has a natural extension in this setting, namely, for all $t \neq 0$, 
\begin{equation}\label{eq:resit-flow}
\mathrm{Resit}(f^t) = \frac{\mathrm{Resit}(f)}{|t|}.
\end{equation}
See Remark \ref{rem:resit-flow} for a proof.
\end{rem}

\begin{rem} In the real or complex setting, if 
$$f(x) = x + ax^{\ell + 1} + bx^{2\ell + 1} + \ldots$$
then, letting $h (x) = x^{\ell}$, one readily checks that  
\begin{equation}\label{eq:red}
h f h^{-1} (x) = x + \ell \, a \, x^2 + \left[ \ell \, b + \frac{ \ell \, (\ell-1)}{2} a^2 \right] x^3 + \ldots.
\end{equation}
Moreover, using Lemmas \ref{l:taylor} and \ref{l:debile} below, one can show that $hfh^{-1}$ is of class $C^3$ if $f$ is of class $C^{2\ell+1}$ (yet this is not crucial to define residues, according to Remark \ref{rem:diferenciable}).  Now, (\ref{eq:red}) easily yields
$$\mathrm{Resad}_1 (h f h^{-1}) = \ell \,\, \mathrm{Resad}_{\ell} (f) 
\qquad \mbox{ and } \qquad 
\mathrm{Resit} (h f h^{-1}) = \frac{\mathrm{Resit}(f)}{\ell}.$$
The value of $\mathrm{Resad}_{\ell}$ hence equals (up to a constant) that of $\mathrm{Resad}_1$ of the conjugate, and the same holds for $\mathrm{Resit}$. 
Since $\mathrm{Resad}_1$ is nothing but a sixth of the Schwarzian derivative, this gives even more insight on $\mathrm{Resad}_{\ell}$ as a generalization of the Schwarzian derivative. 
Compare \cite{BW}, where the Schwarzian derivative naturally arises in the study of conjugacy classes of germs of real-analytic diffeomorphisms.)
\label{rem:i-1}
\end{rem}


\section{On conjugacies and residues: statements, examples}
\label{section-statements}

Let us next recall the first two main results of this work.

\begingroup
\def\thethm{\ref{t:A}}
\begin{thm} \label{t:A-bis} 
 Given $\ell \geq 1$, let $f,g$ be two parabolic germs in $\mathrm{Diff}^{2\ell + 1}_+(\mathbb{R},0)$ that are exactly 
$\ell$-tangent to the identity. If $f$ and $g$ are conjugated by a germ in $\mathrm{Diff}^{\ell + 1}_+ (\mathbb{R},0)$, then they have the same 
(iterative) residue.
\end{thm}
\addtocounter{thm}{-1}
\endgroup

\begingroup
\def\thethm{\ref{t:B}}
\begin{thm} \label{t:B-bis} 
Given $\ell \geq 1$, let $f,g$ be two parabolic germs in $\mathrm{Diff}^{2\ell + 1}_+(\mathbb{R},0)$ 
that are exactly $\ell$-tangent to the identity. If they are both expanding or both contracting,  
then they are conjugated by a germ in $\mathrm{Diff}^{\, \ell}_+ (\mathbb{R},0)$.
\end{thm}
\addtocounter{thm}{-1}
\endgroup

In \S \ref{section-proofs}, we will give two complete and somewhat independent proofs of Theorem \ref{t:A}. A proof of Theorem \ref{t:B} (in the extended form given by Theorem \ref{t:Bp} from the Introduction) will be presented in \S \ref{section-parte2}. Here we illustrate both of them with several clarifying examples;  in particular, in  \S \ref{section-ejemplo-helene}, we give an example that shows that the converse of Theorem \ref{t:A} does not hold. The reader interested in proofs may skip most of this (somewhat long) section. 


\subsection{A fundamental example: (non-)invariance of residues}
\label{sub-fundamental}

Let us consider the germs of 
$$f(x) = x - x^2  \qquad \mbox{ and } \qquad \, \, g(x) = \frac{x}{1+x} = x - x^2 +x^3 - x^4 + \ldots$$ 
These have different residues:
$$\mathrm{Resit} (f) = \mathrm{Resad}_1 (f) = 1 \neq 0 = \mathrm{Resad}_1 (g) = \mathrm{Resit} (g).$$
Using that $\Phi_{1}(f) = \Phi_{1} (g)$, one easily concludes that any $C^2$ conjugacy between $f$ and $g$ has to be parabolic (see the first proof below). 
Now, because of the invariance of $\mathrm{Resad}_1$ (equivalently, of the Schwarzian derivative) under conjugacies by parabolic germs of $C^3$ diffeomorphisms, we conclude that no $C^3$ conjugacy between them can exist. However, Theorems \ref{t:A} and \ref{t:B} above imply that, actually,  these germs are not $C^2$ conjugate, though they are $C^1$ conjugate. Below we elaborate on these two claims. We do this in two different ways: the former by directly looking at the conjugacy relation, and the latter by looking at the associated vector fields.

\vspace{0.3cm}

\begin{proof}[Sketch of proof that $f,g$ are not $C^2$ conjugate.] Assume that $h$ is a germ of $C^2$ diffeomorphism conjugating $f$ and $g$, that is, $hfh^{-1} = g$. Writing $h(x) = \lambda \, x + a x^2 + o(x^2)$ for $\lambda = Dh (0) \neq 0$ and $a = D^2h(0)/2$, equality $hf = gh$ translates to 
$$\lambda \, f (x) + a \, f(x)^2 + o (x^2) = h(x)- h(x)^2+o(x^2).$$
Thus, 
$$\lambda \, (x-x^2) + a \, (x-x^2)^2 + o(x^2) = (\lambda x + ax^2) - (\lambda x +ax^2)^2 + o(x^2).$$
By identifying the coefficients of $x^2$, this yields
$$- \lambda + a = a - \lambda ^2,$$ 
hence $\lambda = \lambda^2$ and, therefore, $\lambda = 1$. 

Now, knowing that $Dh(0)=1$, we deduce that there exists $C > 0$ such that 
\begin{equation}\label{C2-0}
| h(x) - x | \leq C \, x^2
\end{equation}
for small-enough $x> 0$. Now notice that \, $g^n (x) = \frac{x}{1+nx}.$ \,In particular, given any fixed $x_0' > 0$, 
\begin{equation} \label{C2-1}
g^n (x_0') \geq \frac{1}{n+D} 
\end{equation}
for a very large $D > 0$ (namely, for $D \geq 1/x_0'$) and all $n \geq 1$. In what concerns $f$, a straightforward induction argument (that we leave to the reader; see also Proposition 
\ref{prop:deviation}) shows that, for all $x_0 > 0$, there exists a (very small) constant $D' > 0$ such that, for all $n \geq 1$, 
\begin{equation} \label{C2-2}
f^n (x_0) \leq \frac{1}{n + D' \log(n)}.
\end{equation} 
Now fix $x_0 > 0$ and let $x_0' := h (x_0)$. The equality $hf^n (x_0) = g^n h(x_0)$ then yields 
$$hf^n(x_0) - f^n (x_0) = g^n (x_0') - f^n (x_0).$$
Using (\ref{C2-0}), (\ref{C2-1}) and (\ref{C2-2}), we obtain
$$\frac{C}{(n+D'\log(n))^2} \geq C (f^n(x_0))^2 \geq \frac{1}{n+D} - \frac{1}{n + D' \log (n)} = \frac{D' \log (n) - D}{ (n+D) (n + D' \log (n))},$$
which is impossible for a large-enough $n$.
\end{proof}

\vspace{0.3cm}

\begin{proof}[Sketch of proof that $f,g$ are $C^1$ conjugate.] The existence of a $C^1$ conjugacy between diffeomorphisms as $f$ and $g$ above is folklore (see for instance \cite{rog,firmo}). An argument that goes back to Szekeres \cite{szekeres} and Sergeraert \cite{sergeraert} proceeds as follows: The equality $hfh^{-1} = g$ implies $hf^nh^{-1} = g^n$ for all $n \geq 1$, hence $h = g^{-n} h f^n$. Thus, if $h$ is of class $C^1$, then
\begin{equation}\label{eq-lim}
Dh (x) = \frac{Df^n (x)}{Dg^n (g^{-n} h f^n (x))} Dh (f^n (x)) = \frac{Df^n (x)}{Dg^n (h(x))} Dh (f^n (x)).
\end{equation}
It turns out that the sequence of functions 
$$(x,y) \to A_n (x,y) := \frac{Df^n (x)}{Dg^n (y)}.$$
is somewhat well behaved. In particular, it converges to a continuous function $A$ away from the origin. Since $Dh(0)=1$ (see the  proof above), equation (\ref{eq-lim}) translates to 
\begin{equation}\label{Dh=A}
Dh (x) = A (x,h(x)).
\end{equation}
This is an O.D.E. in $h$. Now, a careful analysis shows that this O.D.E. has a solution, and this allows one to obtain the conjugacy between $f$ and $g$. See \cite{firmo} for the details.
\end{proof}

\vspace{0.3cm}

It is worth stressing that both proofs above used not just the conjugacy relation but an iterated version of it:
$$hfh^{-1} = g \quad \Rightarrow \quad hf^nh^{-1} = g^n.$$
In this regard, it may be clarifying for the reader to look for direct proofs just using the conjugacy relation in order to detect where things get stuck. The moral is that one really needs to use the underlying dynamics. Now, there are objects that encode such dynamics, namely, the vector fields whose time-1 maps correspond to the given diffeomorphisms. These were proved to exist by Szekeres \cite{szekeres} and Sergeraert \cite{sergeraert} in a much broader context. Their properties were studied, among others, by Takens in the non-flat $C^{\infty}$ case \cite{takens} and later by Yoccoz in finite regularity \cite{yoccoz}. We next illustrate how to use them with the example above.

\vspace{0.3cm}

\begin{proof}[Proof that $f,g$ are not $C^2$ conjugate using vector fields.]  
One proof for this particular case that is close to the preceding one works as follows: Let $X,Y$ be, respectively, the (unique) $C^1$ vector fields associated to $f,g$ (according to Takens and Yoccoz, these must be of class $C^{\infty}$; actually, the expression for $Y$ is explicit, as shown below). 
For $x$ close to 0, one has (see \S \ref{sub-residues-fields} for $X$):
$$X (x) = -x^2 - x^3 + o(x^3), \qquad Y(x) = -x^2.$$
Extend $X,Y$ to $\R_+ := [0, \infty)$ so that their corresponding flows $(f^t)$ and $(g^t)$ are made of $C^{\infty}$ diffeomorphisms of $\R_+$ and are globally contracting. Fix $x_0 > 0$. The map $\tau_X \!: x\mapsto \int_{x_0}^x\frac1X$ defines a $C^{\infty}$ diffeomorphism from $\R_+^* :=  \, (0,\infty)$ to $\R$ satisfying $\tau_X(f^t(x_0))=t$ by definition of the flow, so that $\tau_X^{-1}$ is the map $t\mapsto f^t(x_0)$. Obviously,  one has a similar construction for $Y$ and $(g^t)$. (We refer to \cite{eynard-navas-sternberg} for further details of this construction.) In particular, using the equalities 
$$\int_{x_0}^{f^n (x_0)} \frac{dy}{X(y)} \,\, = \tau_X (f^n(x_0)) \,\, = \,\, n \,\, = \,\, \tau_Y (g^n(x_0)) \,\, = \,\, \int_{x_0}^{g^n (x_0)} \frac{dy}{Y(y)},$$
one gets equations for the values of $f^n(x_0), g^n(x_0)$ that easily yield the estimates (\ref{C2-1}) and~(\ref{C2-2}). Then the very same arguments as in the preceding proof allow to conclude.

Another (more direct) argument of proof works as follows. If $h$ is a $C^2$ diffeomorphism that conjugates $f$ and $g$, then it must conjugate the associated flows (this is a consequence of the famous Kopell Lemma; see for instance \cite[Proposition 2.2]{eynard-navas-sternberg}). This means that \, $X \cdot Dh = Y \circ h.$ \, Writing  $h (x) = \lambda x + a x^2 + o(x^2)$, with $\lambda \neq 0$,  this gives
\begin{small}
$$(-x^2 - x^3 + o(x^3)) \cdot (\lambda + 2ax + o(x)) =  -h(x)^2 + o(x^3) = - (\lambda x + a x^2 + o(x^2))^2  + o(x^3).$$
\end{small}By identifying the coefficients of $x^2$ above, we obtain  $- \lambda = - \lambda^2$,  hence $\lambda = 1$. Next, by identifying the coefficients of $x^3$, we obtain 
$$-2a-1 = -2a - \lambda = -2a \lambda = - 2a,$$
which is absurd. 
\end{proof}

\vspace{0.12cm}

\begin{proof}[Proof that $f,g$ are $C^1$ conjugate using vector fields.] This is a direct consequence of the much more general result below. For the statement, we say that a (germ of) vector field $Z$ at the origin is {\em contracting} (resp. {\em expanding}) if its flow maps are contracting (resp. expanding) germs of diffeomorphisms for positive times. Equivalently, $Z(x) < 0$ (resp. $Z(x) > 0$) for all small-enough $x > 0$.

In all what follows, whenever $u,v$ are functions of a variable $x$ which are nonzero for $x > 0$, 
we will write $u \sim v$ if they satisfy $\lim_{x \to 0} \frac{ u(x) }{ v(x) } = 1$.

\vspace{0.1cm}

\begin{prop}
\label{p:Takens}
Let $X$ and $Y$ be two (germs of) continuous vector fields, both contracting or both expanding, that generate flows of (germs of) $C^1$  
diffeomorphisms. Suppose that there exist $r>1,s>1$ and $\alpha \neq  0,\beta \neq 0$ of the same sign satisfying $X(x)\sim \alpha x^r$ and 
$Y(x)\sim \beta x^s$ at $0$. Then $X$ and $Y$ are conjugate by the germ of a $C^1$ diffeomorphism if and only if $r=s$.
\end{prop}

\vspace{0.1cm}

\begin{rem}
We will see in the proof that the condition remains necessary if one only assumes $r \geq 1, s \ge 1$. However, it is not sufficient anymore if $r=s=1$. In this case, one has to add the condition $\alpha=\beta$ and, even then, one needs an additional regularity assumption on $X$ and $Y$ (for instance, $C^{1+\tau}$ with $\tau>0$ is enough \cite{MW}, but $C^{1+bv}$ and even $C^{1+ac}$ is not \cite{eynard-navas-sternberg}).
\end{rem}

\begin{rem}
Proposition \ref{p:Takens} gives a very simple criterion of $C^1$ conjugacy for contracting smooth vector fields $X$ and $Y$ \emph{which are neither hyperbolic nor infinitely flat at $0$} (or for $C^k$ vector fields which are not $k$-flat). This criterion is simply to have the same order of flatness at~$0$ and the same ``sign'', or equivalently to satisfy that $X/Y$ has a positive limit at $0$ (though the statement above is much more general, since it does not even require the vector fields to be $C^1$~!). 

We do not know whether this condition remains sufficient if one considers infinitely flat vector fields. Nevertheless, in this context it is not necessary. For example, $X(x):=e^{-\frac1{x}}$ and $Y(x):=\frac12 e^{-\frac1{2x}}$ are conjugate by a homothety of ratio $2$, but $Y(x) = \frac12\sqrt{X(x)}$, so $X/Y$ goes to $0$. It is not necessary even if one restricts to conjugacies by germs of parabolic $C^1$ diffeomorphisms. For example, consider the vector field $X(x):= e^{-\frac1{x^2}}$ and the local (analytical) diffeomorphism $h(x):=\frac{x}{1+x}$. Then $Y=h^*X$ satisfies 
$$Y(x) = \frac{(X\circ h)(x)}{Dh(x)}\sim  e^{-\frac{(1+x)^2}{x^2}} = X(x) \, e^{-\frac2x-1},$$
so $X/Y$ goes to $+\infty$.
\end{rem}

\begin{proof}[Proof of the ``necessity'' in Proposition \ref{p:Takens}.] If $h$ is a germ 
of $C^1$ diffeomorphisms conjugating $X$ to $Y$ then, for $x>0$ close to the origin,
$$Dh(x) = \frac{(Y\circ h)(x)}{X(x)}\sim \frac{\beta \, (h(x))^s}{\alpha x^r}\sim \frac{\beta \, (Dh(0))^s}{\alpha} x^{s-r}.$$
The last expression must have a nonzero limit at $0$, which forces $r = s$. 
Observe that this part of the argument works for any values of $r$ and $s$. Moreover, if $r=s=1$ then, 
since $Dh(x)$ must tend to $Dh(0)\neq0$, we have in addition $\alpha = \beta$.
\end{proof}

\begin{proof}[Proof of the ``sufficiency'' in Proposition \ref{p:Takens}.] Again, extend $X,Y$ to $\R_+ = [0, \infty)$ so that they do not vanish outside the origin and their corresponding flows $(f^t)$ and $(g^t)$ are made of $C^1$ diffeomorphisms of $\R_+$. Multiply them by $-1$ in case they are expanding, fix $x_0 > 0$, and consider the maps $\tau_X \!: x \mapsto \int_{x_0}^x\frac1X$ and $\tau_Y \!: x \mapsto \int_{x_0}^x \frac{1}{Y}$. The $C^1$ diffeomorphism $h := \tau_Y^{-1}\circ \tau_X$ of $\R_+^* \!= \, (0,\infty)$ conjugates $(f^t)$ to $(g^t)$, and thus sends $X$ to $Y$. Let us check that, under the assumption $r=s>1$, the map $h$ extends to a $C^1$ diffeomorphism of $\R_+$. For $x>0$ near $0$, one has 
$$Dh(x) = \frac{(Y\circ h)(x)}{X(x)}\sim \frac{\beta}{\alpha}\left(\frac{h(x)}{x}\right)^{\! r},$$
so it suffices to show that $\frac{h(x)}{x}$ has a limit at $0$. To do this, first observe that the improper integral $\int_{x_0}^x\frac{dy}{y^r}$ diverges when $x$ goes to $0$, so 
$$\tau_X(x) = \int_{x_0}^x\frac1X\sim \int_{x_0}^x \frac{dy}{\alpha y^r}\sim -\alpha' x^{1-r}$$
for some constant $\alpha'>0$. Similarly, $\tau_Y(x)\sim -\beta' x^{1-r}$. It follows that when $t$ goes to $\infty$, one has $t = \tau_Y(\tau_Y^{-1}(t))\sim -\beta' (\tau_Y^{-1}(t))^{1-r}$, so that 
\begin{equation}\label{eq:estrella}
\tau_Y^{-1}(t)\sim -\beta'' t^{\frac1{1-r}}. 
\end{equation}
Finally, 
$$h(x) = \tau_Y^{-1}\tau_X(x)\sim \beta'' (\tau_X(x))^{\frac1{1-r}}\sim \beta''' x$$
for some new constant $\beta'''$, which concludes the proof.
\end{proof}

\begin{rem} The inoffensive relation (\ref{eq:estrella}) above is a key step. It does not work for $r=s=1$. In this case, $x^{1-r}$ must be replaced by $\varphi (x) := \log(x)$, but the function $\varphi$ does not satisfy 
$$u(t)\sim v(t) \quad \Longleftrightarrow \quad \varphi(u(t))\sim \varphi(v(t)),$$
whereas the power functions do (as it was indeed used in the line just following (\ref{eq:estrella}))!
\end{rem}

\vspace{0.1cm}

So far we have discussed the $C^1$ conjugacy between elements in $\mathrm{Diff}^1_+ (\mathbb{R},0)$ that are exactly $\ell$-tangent to the identity for the same value of $\ell$ via two different methods: one of them is direct and the other one is based on associated vector fields.  It turns out, however, that  these approaches are somehow equivalent. Indeed, according to Sergeraert \cite{sergeraert}, the vector field 
associated to a contracting germ $f$ of class $C^2$ 
has an explicit iterative formula:
$$X (x) = \lim_{n \to \infty} (f^n)^* (f-id) ( x ) =  \lim_{n \to \infty} \frac{f^{n+1}(x)- f^n (x)}{Df^n (x)}$$
(we refer to \cite{eynard-navas-arcconnected} for the details and extensions of this construction). Now, if $h$ is a $C^1$ diffeomorphism that conjugates $f$ to another contracting  germ $g$ of $C^2$ diffeomorphism, then, as already mentioned, it must send $X$ to the vector field $Y$ associated to $g$, that is, $X \cdot Dh = Y \circ h$.  Using the iterative formula above, we obtain
$$Dh (x) = \lim_{n \to \infty} \frac{Df^n (x)}{f^{n+1}(x)-f^n(x)} \cdot \frac{g^{n+1}(h(x)) - g^n (h(x))}{Dg^n (h(x))}.$$
Since $hfh^{-1} = g$, we have
$$\frac{g^{n+1}(h(x)) - g^n (h(x))}{f^{n+1}(x) - f^n (x)} = \frac{ h (f^{n+1}(x))- h (f^n (x))}{f^{n+1}(x)-f^n(x)} \longrightarrow Dh (0).$$
As a consequence, if $Dh (0) = 1$, we have 
$$Dh (x) = \lim_{n \to \infty} \frac{Df^n(x)}{Dg^n (h(x))},$$
and we thus retrieve the equality $Dh(x) = A (x,h(x))$ from (\ref{Dh=A}). 
\end{proof}


\subsection{Examples of low regular conjugacies that preserve residues}
\label{section:ejemplo}

The statement of Theorem \ref{t:A} would be empty if all $C^{\ell + 1}$ conjugacies between parabolic germs in $\mathrm{Diff}^{2\ell + 1}_{\ell} (\mathbb{R},0) \setminus \mathrm{Diff}^{2\ell + 1}_{\ell + 1} (\mathbb{R},0)$ were automatically $C^{2\ell + 1}$. However, this is not at all the case. Below we provide the details of a specific example for which $\ell = 1$; we leave the extension to higher order of tangency to the reader.  

Actually, our example directly deals with (flat) vector fields: we will exhibit  
two of them of  class $C^3$ that are $C^2$ conjugate but not $C^3$ conjugate. The announced diffeomorphisms will be the time-1 maps of their flows.

Let $X(x) := x^2$, $h(x) := x+x^2+x^3\log x$ and $Y := h^*X$. One easily checks that $h$ is the germ of $C^2$ diffeomorphism that is not $C^3$ on any neighborhood of the origin. (This follows from the fact that the function $x\mapsto x^{k+1} \log x$ is of class $C^k$ but not $C^{k+1}$, which can be easily checked by induction.) Every map that conjugates $X$ to $Y$ equals $h$ up to a member of the flow of $X$ (this immediately follows from \cite[Lemma 2.6]{eynard-navas-sternberg}). As these members are all of class $C^{\infty}$ and $h$ is not $C^3$, there is no $C^3$ conjugacy from $X$ to $Y$.

We are hence left to show that $Y$ is of class $C^3$. To do this, we compute: 
\begin{eqnarray*}
Y(x) 
&=& \frac{(X\circ h)(x)}{Dh(x)} \\
&=& \frac{(x+x^2+x^3\log x)^2}{1+2x+3x^2\log x+ x^2} \\
&=& x^2 \cdot \frac{1+x^2+x^4(\log x)^2+2x+2x^2\log x+2x^3\log x}{1+2x+3x^2\log x+ x^2}\\
&=& x^2\left(1+\frac{-x^2\log x+2x^3\log x+x^4(\log x)^2}{1+2x+3x^2\log x+ x^2}\right)\\
&=& x^2+x^4\log x \cdot \frac{-1+2x+x^2\log x}{1+2x+3x^2\log x+ x^2}.
\end{eqnarray*}
This shows that $Y$ is of the form $Y(x) = x^2 + o( x^4 \log x)$. Showing that $Y$ is of class $C^3$ (with derivatives equal to those of $x^2$ up to order 3 at the origin) requires some extra computational work based on the fact that the function $u (x) := x^4\log x$ satisfies $D^{(3-n)} u (x) = o(x^{n})$ 
for $n \in [\![0,3]\!] := \{0,1,2,3\}$. We leave the details to the reader.


\subsection{$C^{2\ell+1}$ germs with vanishing residue that are non $C^{\ell+1}$ conjugate}
\label{section-ejemplo-helene}

The aim of this section is to give an example showing that the converse to Theorem \ref{t:A} does not hold (for $\ell = 1$). More precisely, we will give an example of two expanding germs of $C^3$ diffeomorphisms, both exactly 1-tangent to the identity and with vanishing iterative residue, that are not $C^2$ conjugate. Notice that, by Proposition \ref{p:Takens}, these are $C^1$ conjugate.

 Again, our example directly deals with (flat) vector fields. Namely, let $X(x):= x^2$ and $Y(x) := h^* X (x)$, where $h$ is given on $\R_+^*
$ by $h(x) := x + x^2 \, \log(\log x).$ Let $f$ (resp. $g$) denote the time-1 map of $X$ (resp. $Y$). We claim that these are (germs of) $C^3$ diffeomorphisms. This is obvious for $f$ (which is actually real-analytic). For $g$, this follows from the fact that $Y$ is of class $C^3$. To check this, observe that 
 $$Y(x) = \frac{X(h(x))}{Dh (x)} = \frac{ \big( x+x^2 \log (\log x) \big)^2}{ 1 + 2 x \log(\log x) + \frac{x}{\log x} },$$
hence, skipping a few steps, 
\begin{equation}\label{eq-taylor-Y}
Y(x) = x^2 + x^3 \cdot \frac{ x \, (\log (\log x))^2 - \frac{1}{\log x}}{1 + 2 x \log(\log(x)) + \frac{x}{\log(x)} } = x^2 + o(x^3).
\end{equation}
Now, letting 
$$u(x) := x^3, \qquad v(x):= \frac{ x \, (\log (\log x))^2 - \frac{1}{\log x}}{1 + 2 x \log(\log x) + \frac{x}{\log x} },$$ 
we have $Y(x) = x^2 + (u v) (x)$, and  
$$D^{(3)} (uv) (x) = \sum_{j=0}^3 {3 \choose j} \, D^{(3-j)} u (x) \, D^{(j)} (v)(x) = \sum_{j=0}^3 c_j \, x^j \, D^{(j)} v(x)$$
for certain constants $c_j$. We are hence reduced to showing that $x^j D^{(j)}v(x)$ tends to $0$ as $x$ goes to the origin, which is straightforward and is left to the reader. 

\vspace{0.1cm}

It will follow from \S \ref{sub-residues-fields} (see equation (\ref{eq-fnvf})) that both $f$ and $g$ have vanishing iterative residue. 
They are conjugated by $h$, which is the germ of a $C^1$ diffeomorphism which is easily seen to be non $C^2$. As in the previous 
example, this implies that there is no $C^2$ conjugacy between $f$ and $g$ (despite they are both $C^3$ and have vanishing iterative residue\,!).


\section{On the conjugacy invariance of residues}
\label{section-proofs}

In this section, we give two proofs of Theorem \ref{t:A} that generalize those given in the particular case previously treated. 
In most of the arguments of both proofs, a careful study of the associated vector fields is crucial. In particular, a key argument rests on a 
very subtle property, namely, the fact that for non-flat germs of $C^k$ diffeomorphisms, these vector fields are still $k$ times differentiable 
at the origin, despite they may fail to be of class $C^k$ (as mentioned before, they are ensured to be $C^{k-1}$). Quite surprisingly, 
this is not explicitly stated this way in Yoccoz' paper \cite[Appendice 3]{yoccoz}, though it follows directly from  \S 8 therein. 
More precisely, it corresponds to his second condition $(C_i)$ for $i=k-1$.  
(Compare \cite[Lemma 1.1]{eynard-navas-sternberg}, which is the analogous property in the hyperbolic case.)

It is worth pointing out that, however, one of our proofs avoids the use of vector fields. Although this makes it longer, it remains 
completely elementary. It is closely related to classical arguments of Fatou regarding linearizations of parabolic germs of complex 
analytic maps \cite{carlesson-gamelin}; compare \cite{maja}.


\subsection{Residues and vector fields}
\label{sub-residues-fields}

We begin by recalling Yoccoz' result and, in particular, by relating the residue to their infinitesimal expression.

Let $\mathbb{K}$ be a field. 
For each integer $k \geq 1$, let us consider the group  $G_k (\mathbb{K})$ made of the expressions of the form 
$$\sum_{n=1}^k a_n x^n, \quad \mbox{ with } \, a_1 \neq 0,$$
where the product is just the standard composition but neglecting terms of order larger than $k$. The group $\widehat{\mathrm{Diff}} (\mathbb{K},0)$ has a natural morphism into each group $G_k (\mathbb{K})$ obtained by truncating series expansions at order $k$: 
$$\sum_{n \geq 1} a_n x^n \in \widehat{\mathrm{Diff}} (\mathbb{K},0) \quad \longrightarrow \quad \sum_{n=1}^{k} a_n x^n \in G_k (\mathbb{K}).$$
Notice that each group $G_k (\mathbb{K})$ is solvable and finite-dimensional. Moreover, each subgroup $G_{k,\ell} (\mathbb{K})$, $\ell < k$, obtained by truncation as above of elements in $\geri$ is nilpotent. In case where $\mathbb{K}$ is the field of real  numbers, these groups have a well-defined exponential map, and every group element is the time-1 element of a unique flow. 

Actually, it is not very hard to explicitly compute this flow for elements $\mathrm{f} \in G_{2\ell + 1,\ell } (\mathbb{K})$, that is for those of the form 
$$\mathrm{f} \, (x) = x + \sum_{n=\ell + 1}^{2\ell + 1} a_n x^n.$$
Namely, this is given by 
\begin{equation}\label{eq-vf}
\mathrm{f}^t (x) 
= x + \sum_{n=\ell + 1}^{2 \ell} t \, a_n \, x^n + \left[  \frac{\ell + 1}{2} (t \, a_{\ell})^2 - t \, \mathrm{Resad}_{\ell} (\mathrm{f})  \right] x^{2\ell + 1},
\end{equation}
where $\mathrm{Resad}_{\ell}$ is defined as in $\geri$ by (\ref{eq-resad-i})  (assuming that $\mathbb{K}$ has characteristic different from $2$). Checking that this is indeed a flow is straightforward: it only uses the additive properties of the corresponding versions of $\Phi_{\ell,i}$ (for $1\leq i \leq \ell$) and $\mathrm{Resad}_{\ell}$ in $G_{2\ell + 1}(\mathbb{K})$.  

Inspired by the work of Takens \cite{takens}, in \cite[Appendice 3]{yoccoz}, Yoccoz considered germs $f$ of non-flat $C^k$ diffeomorphisms of the real line fixing the origin. In particular, he proved that there exists a unique (germ of) $C^{k-1}$ vector field $X$ that is $k$ times differentiable at the origin and whose flow $(f^t)$ has time-1 map $f^1 = f$. 

If $k = 2\ell + 1$, by truncating $f$ at order $2\ell + 1$, we get an element $\mathrm{f}_{2\ell + 1} \in G_{2\ell + 1} (\mathbb{R})$, and we can hence consider the flow $( \mathrm{f}^t_{2\ell + 1})$ in $G_{2\ell + 1} (\mathbb{R})$ whose time-1 element is  $\mathrm{f}_{2\ell + 1}$. Let $P$ be the associated polynomial vector field defined by
$$P(x) := \frac{d}{dt}_{ |_{t=0}} \mathrm{f}^t_{2\ell + 1} (x).$$ 
As a key step of his proof, Yoccoz showed a general estimate that implies that 
$$\lim_{x \to 0} \frac{X(x) - P(x)}{x^{2\ell + 1}} = 0.$$

According to (\ref{eq-vf}), if $f \in \mathrm{Diff}^{2\ell + 1}_{\ell } (\mathbb{R},0)$ writes in the form 
$$f(x) = x + \sum_{n=\ell + 1}^{2\ell + 1} a_n x^n + o (x^{2\ell + 1}),$$
then the polynomial vector field associated to $\mathrm{f}_{2\ell + 1}$ equals 
$$P(x) = \frac{d}{dt}_{ |_{t=0}} \mathrm{f}^t_{2\ell + 1} (x) = \sum_{n=\ell + 1}^{2 \ell} a_n x^n - \mathrm{Resad}_{\ell}(f) \, x^{2\ell + 1},$$
and therefore  
\begin{equation}\label{eq-fnvf}
X (x) = \sum_{n=\ell + 1}^{2\ell} a_n x^n - \mathrm{Resad}_{\ell}(f) \, x^{2\ell + 1} + o(x^{2\ell + 1}).
\end{equation} 
This formula in which $\mathrm{Resad}_{\ell}$ explicitly appears will be fundamental for the proof of Theorem~\ref{t:A}.

\begin{rem} 
\label{rem:resit-flow}
Formula (\ref{eq:resit-flow}), which holds for every $f \in \mathrm{Diff}^{2\ell + 2}_{\ell} (\mathbb{R},0) \setminus \mathrm{Diff}^{2\ell + 2}_{\ell+1} (\mathbb{R},0)$ and all $t \neq 0$, follows from (\ref{eq-fnvf}) above. Indeed, in order to check it, we may assume that $f$ is expanding and already in reduced form, say 
$$f (x) = x + x^{\ell+1} + \mu x^{2\ell+1} + o(x^{2\ell + 1}).$$
If $X$ is the $C^{2\ell+1}$ vector field associated to $f$, then the vector field $X_t$ associated to $f^t$ is $t X$. Since 
$$X_t (x) = t \, X(x) = 
t \left[  x^{\ell+1} - \mathrm{Resad}_{\ell}(f) \, x^{2\ell + 1} + o(x^{2\ell + 1}) \right],$$
by (\ref{eq-fnvf}) we have 
$$f^t (x) = x + t \, x^{\ell+1} + \mu_t \, x^{2\ell +1} + o(x^{2\ell +1}),$$
where $\mu_t$ is such that $\mathrm{Resad}_{\ell} (f^t) = t \, \mathrm{Resad}_{\ell} (f)$. Hence, 
$$\frac{(\ell +1) \, t^2}{2} - \mu_t = t \, \left[ \frac{\ell + 1}{2} - \mu \right].$$
Therefore, for $t > 0$,
$$\mathrm{Resit} (f^t) = \frac{\ell +1}{2} - \frac{\mu_t}{t^2} 
= \frac{1}{t^2} \left[ \frac{(\ell +1) \, t^2}{2} - \mu_t \right] = \frac{t}{t^2} \, \left[ \frac{\ell + 1}{2} - \mu \right] 
= \frac{1}{t} \, \left[ \frac{\ell + 1}{2} - \mu \right] = \frac{\mathrm{Resit} (f)}{t},$$
as announced. The case $t < 0$ is analogous.

It is worth mentioning that the argument above shows that equality (\ref{eq:resit-flow}) holds for every 
$f \in \mathrm{Diff}^{2\ell + 1}_{\ell} (\mathbb{R},0) \setminus \mathrm{Diff}^{2\ell + 1}_{\ell+1} (\mathbb{R},0)$ and all $t  \neq 0$ for which $f^t$ is a germ of $C^{2\ell+1}$ diffeomorphism (or, at least, has $(2\ell+1)$ derivatives at the origin; see Remark \ref{rem:diferenciable}).
\end{rem}


\subsection{A first proof of the $C^{\ell + 1}$ conjugacy invariance of $\mathrm{Resit}$ }

We next proceed to the proof of the $C^{\ell + 1}$ conjugacy invariance of $\mathrm{Resit}$ between germs of $C^{2\ell + 1}$ diffeomorphisms that are exactly $\ell$-tangent to the identity.

\vspace{0.2cm}

\begin{proof}[Proof of Theorem \ref{t:A}] Let $f$ and $g$ be elements in 
$\mathrm{Diff}^{2\ell + 1}_{\ell} (\mathbb{R},0) \setminus \mathrm{Diff}^{2\ell + 1}_{\ell + 1} (\mathbb{R},0)$ conjugated by an element $h \in \mathrm{Diff}^{\ell + 1}_+ (\mathbb{R},0)$. Then they are both contracting or both expanding. We will suppose that the second case holds; the other one follows from it by passing to inverses. 

As we want to show the equality $\mathrm{Resit} (f) = \mathrm{Resit} (g)$, by the definition and Lemma \ref{lem-basic}, we may assume that $f$ and $g$ have Taylor series expansions at the origin of the form
$$f(x) = x + x^{\ell + 1} + \mu \, x^{2\ell + 1} + o (x^{2\ell + 1}), 
\qquad 
g(x) = x + x^{\ell + 1} + \mu' \, x^{2\ell + 1} + o(x^{2\ell + 1}). 
$$  
Notice that 
$$\mathrm{Resit} (f) = \mathrm{Resad}_{\ell}(f) 
= \frac{\ell + 1}{2} - \mu =: R, \qquad \mathrm{Resit} (g) = \mathrm{Resad}_{\ell} (g) = \frac{\ell + 1}{2} - \mu' =: R'.$$
In virtue of (\ref{eq-fnvf}), the vector fields $X,Y$ associated to $f,g$, respectively, have the form 
$$X(x) = x^{\ell + 1} - R \, x^{2\ell + 1} + o(x^{2\ell + 1}), 
\qquad 
Y(x) = x^{\ell + 1} - R' \, x^{2\ell + 1} + o(x^{2\ell + 1}).$$
Now write 
$$h(x) = \lambda \, x + \sum_{n =2}^{\ell + 1} c_n x^n + o(x^{\ell + 1}).$$ 
Since $h$ conjugates $f$ to $g$, it must conjugate $X$ to $Y$, that is, 
\begin{equation}\label{eq-inv-vf}
X \cdot Dh = Y\circ h.
\end{equation}
We first claim that this implies $\lambda = 1$. Indeed, the relation writes in a summarized way as 
$$(x^{\ell + 1} + o (x^{\ell + 1})) \cdot (\lambda + o(1)) = (\lambda \, x + o (x))^{\ell + 1} + o(x^{\ell + 1}).$$
By identification of the coefficients of $x^{\ell + 1}$, we obtain $\lambda = \lambda^{\ell + 1}$. Thus, $\lambda = 1$, as announced.

We next claim that $h$ must be $\ell$-tangent to the identity. Indeed, assume otherwise and let 
$2 \leq p < \ell + 1$ be the smallest index for which $c_p \neq 0$. Then relation (\ref{eq-inv-vf}) above 
may be summarized as 
$$(x^{\ell + 1} + o(x^{\ell + 1})) \cdot (1 + p \, c_p \, x^{p-1} + o(x^{p-1})) = 
(x + c_p \, x^p + o(x^p))^{\ell + 1}  + o(x^{2\ell}).$$
By identifying the coefficients of $x^{\ell+p}$ we obtain $p \, c_p = (\ell + 1) \, c_p$, which is impossible for $c_p \neq 0$.

Thus, $h$ writes in the form  
$$h(x) = x + cx^{\ell + 1} + o(x^{\ell + 1}).$$ 
Relation (\ref{eq-inv-vf}) then becomes 
\begin{small}
$$( x^{\ell + 1} - R x^{2\ell + 1} + o (x^{2\ell + 1})) \cdot (1 + (\ell + 1) c x^{\ell} + o(x^{\ell})) 
= 
(x + cx^{\ell + 1} + o(x^{\ell + 1}))^{\ell + 1} - R' (x + o(x^{\ell}))^{2\ell + 1} + o(x^{2\ell + 1}).$$
\end{small}Identification of the coefficients of $x^{2\ell + 1}$ then gives 
$$(\ell + 1) \, c - R = (\ell + 1) \, c - R'.$$ 
Therefore, $R=R'$, as we wanted to show. 
\end{proof}


\subsection{Residues and logarithmic deviations of orbits}

We next characterize $\mathrm{Resit}$ for contracting germs 
$f \in \mathrm{Diff}^{2\ell + 1}_{\ell} (\mathbb{R},0) \setminus  \mathrm{Diff}^{2\ell + 1}_{\ell + 1} (\mathbb{R},0)$ 
in terms of the deviation of orbits from those of the corresponding (parabolic) ramified affine flow of order $\ell$, namely, 
$$f^t(x) := \frac{x}{\sqrt[\ell]{1 + t x^{\ell}}} 
= x - \frac{t}{\ell} x^{\ell + 1} + \frac{1}{2 \, \ell} \left( 1+\frac{1}{\ell} \right) t^2 x^{2\ell + 1} + o (x^{2\ell + 1}).$$
Notice that the time-1 map $f := f^1$ of this flow satisfies $\mathrm{Resit} (f) = 0$. 
Moreover, for all $x>0$, 
$$ \frac{1}{\sqrt[\ell]{n}} - f^n (x) 
=  \frac{1}{\sqrt[\ell]{n}} - \frac{x}{\sqrt[\ell]{1+nx^{\ell}}}$$
which asymtotically behaves (when $n$ goes to infinity) as  
$$\frac{n^{\frac{\ell-1}{\ell}}}{\ell \, n \, (1+nx^{\ell})} 
\sim \frac{1}{\ell \, n \, \sqrt[\ell]{n} \, x^\ell},$$
where the first equivalence follows from the equality \, $u^{\ell} - v^{\ell} = (u-v) (u^{\ell-1} + u^{\ell-2}v + \ldots + v^{\ell-1})$. 
In particular, letting $a = 1/\ell$, we have
$$\lim_{n \to \infty} \left[ \frac{ \ell^2 \, n \, \sqrt[\ell]{a \, \ell \, n} }{\log (n)} \left( \frac{1}{ \sqrt[\ell]{a \, \ell \, n}} - f^n (x)\right) \right] = 0.$$
This is a particular case of the general proposition below.

\vspace{0.1cm}

\begin{prop}  \label{prop:deviation}
If $\ell \geq 1$ then, for 
every contracting germ $f \in \mathrm{Diff}^{2\ell + 1}_{\ell} (\mathbb{R},0) \setminus \mathrm{Diff}^{2\ell + 1}_{\ell + 1} (\mathbb{R},0)$ 
of the form
$$f(x) = x - ax^{\ell + 1} + b x^{2\ell + 1} + o(x^{2\ell + 1}), \qquad a > 0,$$
and all $x_0 > 0$, one has
\begin{small}
\begin{equation}\label{eq-resit-desv-gen}
\mathrm{Resit} (f) 
= \lim_{n \to \infty} \left[ \frac{a \, \ell^2 \, n^2}{\log (n)} \left( \frac{1}{a \, \ell \, n} -[ f^n (x_0)]^{\ell}\right) \right] 
= \lim_{n \to \infty} \left[ \frac{ \ell^2 \, n \, \sqrt[\ell]{a \, \ell \, n} }{\log (n)} \left( \frac{1}{ \sqrt[\ell]{a \, \ell \, n}} - f^n (x_0)\right) \right].
\end{equation}
\end{small}
\end{prop}

\vspace{0.1cm} 

\begin{proof}[First proof (using coordinates at infinity).]  
Consider the map $I (z) = 1/z^{1 / \ell}$, where $z > 0$ is large-enough. Then the conjugate germ at infinity $g := I^{-1} f I$ has an expansion of the form
$$g(z) 
= \left( \frac{1}{\frac{1}{z^{1/\ell}} - \frac{a}{z^{(\ell+1)/\ell}} + \frac{b}{z^{(2\ell+1)/\ell}} + ...} \right)^{\ell} 
= z \left( \frac{1}{1 - \frac{a}{z} + \frac{b}{z^2} + \ldots} \right)^{\ell}.$$
Hence,
$$g(z) 
= z \left( 1 + \left(\frac{a}{z} - \frac{b}{z^2} - \ldots \right) + \left(\frac{a}{z} - \frac{b}{z^2} - \ldots \right)^2 + \ldots \right)^{\ell} 
= z \left( 1 + \frac{a}{z} + \frac{a^2 - b}{z^2} + O \left( \frac{1}{z^3} \right) \right)^{\ell},$$
that is 
$$g(z) = z \left( 1 + \ell \left( \frac{a}{z} + \frac{a^2 - b}{z^2} \right) + \frac{\ell (\ell-1)}{2} \left( \frac{a}{z} \right)^2 + O \left( \frac{1}{z^3}\right) \right).$$
Therefore,
\begin{equation}\label{eq-ln}
g(z) 
= z + a\ell + \frac{\ell}{z} \left[ \frac{(\ell + 1) a^2}{2} - b \right] +  O \left( \frac{1}{z^2}\right) 
= z + a\ell + \frac{R \ell}{z} +  O \left( \frac{1}{z^2}\right),
\end{equation}
where $R = \mathrm{Resad}_{\ell} (f)$. 
We claim that, for each fixed $z$, the sequence below converges as $n$ goes to infinity:
\begin{equation}\label{eq-seq}
\varphi_n (z) := g^n (z) - a\ell n - \frac{R  \log(n)}{a}.
\end{equation}
Assume this for a while. Then we have 
$$R 
= \lim_{n \to \infty} \frac{a}{ \log(n)} \big[ g^n(I^{-1}(x_0)) - a \ell n \big],$$
hence 
$$R
=  \lim_{n \to \infty} \frac{a}{ \log(n)  } \big[ I^{-1} f^n (x_0) - a \ell n \big]
=  \lim_{n \to \infty} \frac{a}{ \log(n) } \left[ \frac{1}{f^n (x_0)^{\ell}} - a \ell n \right].$$
This easily implies that 
\begin{equation}\label{eq-ln-encore}
\lim_{n \to \infty} (a \ell n)^{1/\ell} f^n (x_0) = 1.
\end{equation}
Moreover, using the identity \, 
$u^{\ell} - v^{\ell} = (u-v) (u^{\ell-1} + u^{\ell-2}v + \ldots + v^{\ell-1})$, \, one easily deduces that 
$$R 
=  \lim_{n \to \infty} \frac{ a \ell (a \ell n)^{(\ell-1)/\ell}}{\log(n)} \left[ \frac{1}{f^n (x_0)} - (a \ell n)^{1/\ell} \right] 
=  \lim_{n \to \infty} \frac{ a \ell (a \ell n)^{(\ell-1)/\ell}}{\log(n)} \left[ \frac{1}{ (a \ell n)^{1/\ell}} - f^{n}(x_0) \right] \frac{(a \ell n)^{1/\ell}}{f^n(x_0)}.$$
Using (\ref{eq-ln-encore}), one thus concludes 
$$R =  \lim_{n \to \infty} \frac{  a \textcolor{red}{\ell} (a \ell n)^{(\ell - 1)/\ell} (a \ell n)^{2/\ell} }{\log(n)} \left[ \frac{1}{ (a \ell n)^{1/\ell}} - f^{n}(x_0) \right] .$$
Therefore,
$$\mathrm{Resit} (f) 
= \frac{R}{a^2} 
= \lim_{n \to \infty} \frac{ \ell^2 n  (a \ell n)^{1/\ell} }{\log(n)} \left[ \frac{1}{ (a \ell n)^{1/\ell}} - f^{n}(x_0) \right],$$
as announced.

It remains to show that the sequence $\varphi_n(z)$ defined by (\ref{eq-seq}) 
converges as $n$ goes to infinity. First notice that, for all large-enough $z$, 
$$z + \frac{a \ell}{2} \leq g(z) \leq z + 2a\ell.$$
Hence, for a fixed $z$ and large-enough $n$,  
$$\frac{a\ell n}{2} \leq g^n(z) \leq 2 a\ell n.$$
Using (\ref{eq-ln}) and this last estimate, we obtain
\begin{eqnarray*}
\varphi_{n+1} (z) - \varphi_n (z) 
&=& g^{n+1} (z) -a\ell (n+1) - \frac{R \log(n+1)}{a} - \left[ g^n(z) - a\ell n - \frac{R \log(n)}{a} \right] \\
&=& \left[ g^n(z) + a\ell + \frac{R \ell}{g^n(z)} + O \left(\frac{1}{g^n(z)^2} \right) \right] - a\ell - \frac{R }{a} \log \left( \frac{n+1}{n} \right) - g^n(z),
\end{eqnarray*}
hence
\begin{equation}\label{eq-ln3}
\varphi_{n+1} (z) - \varphi_n (z) 
= \frac{R \ell}{g^n(z)} - \frac{R }{a} \log \left( \frac{n+1}{n} \right) + O \left(\frac{1}{g^n(z)^2} \right) 
= O \left( \frac{1}{n} \right).
\end{equation}
This estimate is still weak to prove the announced convergence, but reintroducing it in the computations above will give the desired convergence. 
Indeed, letting $\varphi_0 (z) := z$, it yields
$$|\varphi_n (z) - z| \leq \sum_{i=1}^n |\varphi_i (z) - \varphi_{i-1} (z)| = O (\log(n)).$$
Using the left-hand side equality in (\ref{eq-ln3}), the definition of $\varphi_n$ and the latter estimate, we finally obtain 
\begin{eqnarray*}
\varphi_{n+1} (z) - \varphi_n (z)  
&=& \frac{R \ell}{g^n(z)} -  \frac{R }{a} \log \left( \frac{n+1}{n} \right)  + O \left( \frac{1}{n^2} \right) \\
&=& \frac{R \ell}{a\ell n + \varphi_n (z) + R \log(n) / a} - \frac{R}{an} + O \left( \frac{1}{n^2} \right) \\
&=& \frac{ -R \, [\varphi_n (z) + R \log(n) / a]}{an (a\ell n + \varphi_n (z) + R \log(n) / a)}  + O \left( \frac{1}{n^2} \right) \\
&=& O \left( \frac{\log (n)}{n^2} \right).
\end{eqnarray*}
Since the sum $\sum \frac{\log(n)}{ n^2}$ is finite, this shows that $(\varphi_n(z))$ is a Cauchy sequence, 
which implies the announced convergence.
\end{proof}

\vspace{0.2cm}

\begin{proof}[Second proof (using vector fields).]  
We know from \S \ref{sub-residues-fields} that the vector field $X$ associated to $f$ 
has a Taylor series expansion of the form 
$$X(x) = -a x^{\ell + 1} - R x^{2\ell + 1} + o(x^{2\ell + 1}), \qquad R = \mathrm{Resad}_{\ell} (f).$$
We compute
\begin{eqnarray*}
n 
&=& \int_{x_0}^{f^n(x_0)} \frac{dy}{X(y)} \\
&=&  \int_{x_0}^{f^n(x_0)}  \frac{dy}{-a y^{\ell + 1} - R y^{2\ell + 1}} +  \int_{x_0}^{f^n(x_0)}  \left[ \frac{1}{a y^{\ell + 1} + R y^{2\ell + 1}} - \frac{1}{a y^{\ell + 1} + R y^{2\ell + 1} + o(y^{2\ell + 1})}  \right] \, dy \\
&=&  \int_{x_0}^{f^n(x_0)}  \frac{dy}{-a y^{\ell + 1} (1 + R y^{\ell} /a)} +  \int_{x_0}^{f^n(x_0)}  \left[ \frac{o(y^{2\ell + 1})}{y^{2\ell+2} (a^2 + o(y))}  \right] \, dy \\
&=&  \int_{x_0}^{f^n(x_0)}  \left[ \frac{1 - Ry^{\ell}/a + o(y^{2\ell-1})}{-a y^{\ell + 1}} \right] \, dy +  \int_{x_0}^{f^n(x_0)}  o \left( \frac{1}{y} \right)  \, dy \\
&=& \frac{1}{a \ell y^{\ell}}\Big|^{f^n(x_0)}_{x_0} +  \frac{R \, \log(y)}{a^2}\Big|^{f^n(x_0)}_{x_0} - \frac1a \int_{x_0}^{f^n (x_0)} o \left( y^{\ell-2} \right) + o (\log (f^n (x_0))) \\
&=& \frac{1}{a \, \ell \, [f^n (x_0)]^{\ell}} + \frac{R \, \log(f^n (x_0))}{a^2} + C_{x_{0}} + o (\log (f^n(x_0))),
\end{eqnarray*}
where $C_{x_0}$ is a constant that depends only on $x_0$ (and is independent of $n$). The right-side expression is of the form
$$\frac{1}{a \, \ell \, [f^n (x_0)]^{\ell}} +  o \left( \frac{1}{[f^n(x_0)]^\ell} \right),$$
which shows that, as $n$ goes to infinity,
\begin{equation}\label{eq:1/n}
[f^n (x_0)]^\ell  \sim \frac{1}{a \, \ell \, n}. 
\end{equation}
Now, from
\begin{eqnarray*}
R 
&=& \frac{a^2 \, (n-C_{x_0})}{\log (f^n (x_0))} - \frac{a}{\ell \, [f^n(x_0)]^\ell \, \log (f^n(x_0))}  + o(1) \\
&=& \frac{a^2 \, (n-C_{x_0})}{[f^n(x_0)]^\ell \, \log (f^n(x_0))} \left[ [f^n(x_0)]^\ell - \frac{1}{a \, \ell \,  (n-C_{x_0})} \right] + o(1),
\end{eqnarray*}
using (\ref{eq:1/n}) we obtain 
$$R 
= \lim_{n \to \infty} \frac{a^2 \, (n-C_{x_0})}{\frac{1}{a \, \ell \, n} \cdot \frac{1}{\ell}\log (\frac{1}{a \, \ell \, n})} 
\left[ [f^n(x_0) ]^\ell - \frac{1}{a \, \ell \, (n-C_{x_0})} \right] 
= \lim_{n \to \infty} \frac{a^3 \, \ell^2 \, n^2}{\log(n)} \left[ \frac{1}{a \, \ell \, n} - [f^n (x_0)]^\ell \right].$$
Therefore,
$$\mathrm{Resit}(f) = \frac{R}{a^2} = \lim_{n \to \infty} \frac{a \, \ell^2 \, n^2}{\log(n)} \left[ \frac{1}{a \, \ell \, n} - [ f^n (x_0) ]^\ell \right]
= \lim_{n \to \infty} \left[ \frac{ \ell^2 \, n \sqrt[\ell]{a \ell n}}{\log (n)} \left( \frac{1}{ \sqrt[\ell]{a \, \ell \, n}} - f^n (x_0)\right) \right],$$
where the last equality follows from the identity \, $u^\ell - v^\ell = (u-v) (u^{\ell-1} + u^{\ell-2}v + \ldots + v^{\ell-1})$.
\end{proof}

\vspace{0.1cm}

\begin{rem} 
It is a nice exercise to deduce the general case of the previous proposition from the case $\ell=1$ via conjugacy by the map $x \mapsto x^{\ell}$; see Remark \ref{rem:i-1}.
\end{rem}

\vspace{0.1cm}

What follows is a direct consequence of the previous proposition; checking the details is left to the reader. It is worth comparing the case $\ell=1$ of the statement with (\ref{C2-2}).

\vspace{0.1cm}

\begin{cor}  
If $\ell \geq 1$ then, for every contracting germ $f \in \mathrm{Diff}^{2\ell + 1}_\ell (\mathbb{R},0) \setminus \mathrm{Diff}^{2\ell + 1}_{\ell + 1} (\mathbb{R},0)$ of the form
$$f(x) = x - ax^{\ell + 1} + b x^{2\ell + 1} + o(x^{2\ell + 1})$$
and all $x_0 > 0$, one has
\begin{equation}\label{eq:asymp}
f^n (x_0) = \frac{1}{\sqrt[\ell]{a \ell n}} \left( 1 - \frac{[\mathrm{Resit}(f) + \delta_n] \log (n)}{\ell^2 \, n} \right),
\end{equation}
where $\delta_n$ is a sequence that converges to $0$ as $n$ goes to infinity.
\end{cor}

\subsection{A second proof of the $C^{\ell + 1}$ conjugacy invariance of $\mathrm{Resit}$}  

We can give an alternative proof of Theorem \ref{t:A} using the estimates of the previous section.

\vspace{0.1cm}

\begin{proof}[Second proof of  Theorem \ref{t:A}] 
Let $f,g$ in $\mathrm{Diff}^{2\ell + 1}_{\ell} (\mathbb{R},0) \setminus \mathrm{Diff}^{2\ell + 1}_{\ell + 1} (\mathbb{R},0)$ be elements conjugated by $h \in \mathrm{Diff}^{\ell + 1}_+ (\mathbb{R},0)$. In order to show that $\mathrm{Resit} (f)$ and $\mathrm{Resit} (g)$ coincide, we may assume that $f$ and $g$ have Taylor series expansions at the origin of the form
$$f(x) = x - x^{\ell + 1} + \mu \, x^{2\ell + 1} + o (x^{2\ell + 1}), 
\qquad 
g(x) = x - x^{\ell + 1} + \mu' \, x^{2\ell + 1} + o(x^{2\ell + 1}). 
$$  
The very same arguments of the beginning of the first proof show that $h$ must write in the form 
$$h(x) = x + cx^{\ell + 1} + o(x^{\ell + 1}).$$
In particular, for a certain constant $C > 0$, 
\begin{equation}\label{eq:C}
|h(x) - x| \leq C \, x^{\ell + 1}
\end{equation}
for all small-enough $x > 0$. Fix such an $x_0 > 0$. From $hf = gh$ we obtain $hf^n (x_0) = g^n h(x_0)$ for all $n$, which yields 
$$h f^n (x_0) - f^n(x_0) = g^n (h(x_0)) - f^n(x_0) $$
Using (\ref{eq:asymp}) and (\ref{eq:C}), this implies
\begin{equation}\label{eq:imp}
C \left[ \frac{1}{\sqrt[\ell]{a\ell n}} \left( 1 - \frac{[\mathrm{Resit} (f) + \delta_n] \log (n)}{\ell^2 \, n} \right) \right]^{\ell + 1} 
\geq \left|  \frac{[\mathrm{Resit}(g) - \mathrm{Resit}(f) + \delta_n'] \log (n)}{\ell^2 \, n \, \sqrt[\ell]{a \ell n}}  \right|
\end{equation}
for certain sequences $\delta_n,\delta_n'$ converging to $0$. On the one hand, the left-hand term above is of order $1/ (n \, \sqrt[\ell]{n})$. On the other hand, if $\mathrm{Resit}(f) \neq \mathrm{Resit}(g)$, then the right-hand term is of order 
$$\frac{|\mathrm{Resit}(g) - \mathrm{Resit}(f)| \, \log(n)}{n \, \sqrt[\ell]{n}}.$$
Therefore, if $\mathrm{Resit}(f) \neq \mathrm{Resit}(g)$, then inequality (\ref{eq:imp}) is impossible for large-enough $n$. Thus, $\mathrm{Resit}(f)$ and $\mathrm{Resit}(g)$ must coincide in case of $C^{\ell + 1}$ conjugacy.
\end{proof}


\section{Existence of low regular conjugacies between germs}
\label{section-parte2}

The goal of this section is to prove Theorem \ref{t:Bp}, of which Theorem \ref{t:B} is a particular case (namely, the case $r = \ell$). To begin with, notice that, passing to inverses if necessary, we may assume that the diffeomorphisms in consideration are both expanding. Looking at the associated vector fields, one readily checks that Theorem \ref{t:Bp}  is a direct consequence of the next proposition. For the statement, given $\ell \geq 1$, we say that a (germ of) vector field $Z$ is {\em exactly $\ell$-flat} if it has a Taylor series expansion at the origin of the form
$$Z(x) = \alpha x^{\ell + 1} + o(x^{\ell + 1}), \quad \mbox{with } \alpha \neq 0.$$

\vspace{0.1cm}

\begin{prop} \label{prop:CS}
Given $\ell \ge1$ and $1\le r\le \ell$, any two germs at $0$ of exactly $\ell$-flat, expanding $C^{\ell+r}$ vector fields are $C^{r}$ conjugate.
\end{prop}

\vspace{0.1cm}

For $r=1$, this is a result concerning $C^1$ conjugacies between vector fields that should be compared to Proposition \ref{p:Takens}, yet the hypothesis here are much stronger.

\begin{rem} 
The conjugacy is no better than $C^{r}$ in general. For instance, in \S \ref{section-ejemplo-helene} there is an example of germs of exactly 1-flat, expanding, $C^3$ vector fields (case $\ell=1,$ $r=1$) that are $C^1$ conjugate but not $C^2$ conjugate, despite the time-1 maps have the same (actually, vanishing) iterative residue. It would be interesting to exhibit examples showing that the proposition is optimal for all the cases it covers.
\end{rem}

To prove Proposition \ref{prop:CS}, we will use a well-known general lemma concerning reduced 
forms for vector fields that is a kind of simpler version of \cite[Proposition 2.3]{takens} in finite regularity. (Compare Lemma \ref{lem-basic}, which deals with the case of diffeomorphisms.)

\vspace{0.2cm}

\begin{lem} 
\label{l:reduc}
Let $\ell \geq 1$ and $r \geq 1$, and let $Y$ be a germ of $C^{\ell + r}$ vector field admitting a Taylor series expansion of order $\ell + r$ 
at the origin that starts with $\alpha x^{\ell+1}$, where $\alpha > 0$. If $1\le r \le \ell$, then $Y$ is $C^{\infty}$ conjugate to a vector field 
of the form $x\mapsto x^{\ell+1}+o(x^{\ell+r})$.
\end{lem}

\begin{proof}
For $1\le r \le \ell$, we construct by induction a sequence of $C^\infty$ conjugates $Y_1$, \dots, $Y_r$ of $Y$ such that $Y_s(x)=x^{\ell+1}+o(x^{\ell+s})$ for every $s \in [\![1,r]\!] := \{1,2,\ldots,r\}$. One first obtains $Y_1$ by conjugating $Y$ by a homothety. Assume we have constructed $Y_s$ for some $s \in [\![1,r-1]\!]$, and consider the diffeomorphism $h(x)=x+ax^{s+1}$. Since $Y_s$ is $C^{\ell+s+1}$, there exists $\alpha \in \R$ such that $Y_s(x) = x^{\ell + 1}+\alpha_s x^{\ell +s+1} + o(x^{\ell +s+1})$. Then, for $Y_{s+1}(x) := h^* Y_s(x) =  \frac{Y_s(h(x))}{Dh(x)}$, we have 
\begin{eqnarray*}
Y_{s+1}(x)
&=& \frac{(h(x))^{\ell+1}+\alpha_s (h(x))^{\ell +s+1}+o((h(x))^{\ell +s+1})}{1+a(s+1)x^{s}}\\
&=& \frac{(x+ax^{s+1})^{\ell+1}+\alpha_s (x+ax^{s+1})^{\ell +s+1}+o(x^{\ell+s+1})}{1+a(s+1)x^{s}}\\
&=& \left(x^{\ell +1}+((\ell+1)a+\alpha_s)x^{\ell +s+1}+o(x^{\ell +s+1})\right)\left(1-a(s+1)x^{s}+o(x^{s})\right)\\
&=& x^{\ell+1}+\left((\ell - s)a+\alpha_s \right)x^{\ell+s+1}+o(x^{\ell +s+1}).
\end{eqnarray*}
Since $s \leq r-1 < \ell$, we may let $a := \frac{\alpha_s}{s-\ell}$ to obtain $Y_{s+1}(x) = x^{\ell +1}+o(x^{\ell +s+1})$, which concludes the induction and thus finishes the proof. Observe that the conjugacy that arises at the end of the inductive process is a germ of polynomial diffeomorphism.
\end{proof}

\vspace{0.1cm}

We will also need a couple of elementary technical lemmas. For the statement, in analogy to the case of vector fields, we say that a real-valued function $u$ is {\em $\ell$-flat at $0$} 
if it has a Taylor series expansion at the origin of the form \, $u(x) = a x^{\ell+1} + o (x^{\ell+1}).$

\vspace{0.1cm}

\begin{lem}
\label{l:taylor} 
If $u$ is a $C^{k}$ germ of $m$-flat function at $0$, with $0 \leq m < k$, then  $x\mapsto \frac{u(x)}{x^{m+1}}$ 
extends at $0$ to a germ of $C^{k-m-1}$ map.
\end{lem}

\begin{proof} 
Since $u$ is of class $C^{k}$ and $m$-flat, all its derivatives at the origin up to order $m$ must vanish. Hence, since $u$ is of class $C^{m+1}$, we have the following form of Taylor's formula: 
$$u(x) 
= 0+0\cdot x +\dots + 0\cdot x^m +\frac{x^{m+1}}{m \, !} \int_0^1(1-t)^m D^{(m+1)} u(tx)dt. 
$$
The lemma then follows by differentiating under the integral. 
\end{proof}



\vspace{0.01cm}

\begin{lem}
\label{l:debile}
Let $r\ge 1$ and let $\varphi$ be a germ of continuous function at the origin that is differentiable outside $0$ and such that $\varphi(0)=1$. If $x\mapsto x \, \varphi(x)$ is of class $C^r$, so is the function $x\mapsto x \, (\varphi(x))^\exponente$ for any $\exponente \neq 0$.
\end{lem}

\begin{proof}
Let $u(x) := x \, \varphi(x)$ and $v(x) := x \, (\varphi(x))^\exponente$. Assume $u$ is of class $C^r$. 
If $\exponente = 1$, there is nothing to prove, so let us assume $\exponente \neq 1$. By Lemma \ref{l:taylor}, $\varphi$ is $C^{r-1}$. Now, outside $0$, 
$$Du (x) = \varphi(x) + x \, D\varphi (x)\quad\mbox{ and }\quad Dv (x) = (\varphi(x))^\exponente + \exponente \, x\, D\varphi(x) \, (\varphi(x))^{\exponente-1}.$$
By the first equality, since $Du$ and $\varphi$ are $C^{r-1}$, so is $x\mapsto x \, D \varphi (x)$. Thus, by the second equality, $Dv$ 
writes as a sum of products and compositions of $C^{r-1}$ maps, hence it is $C^{r-1}$. Therefore, $v$ is $C^r$, as announced.
\end{proof}

\vspace{0.1cm}

\begin{proof}[Proof of Theorem \ref{t:Bp}] 
Let $f$ be an expanding germ in $\mathrm{Diff}^{\ell + 1 + r}_+(\mathbb{R},0)$ that is exactly $\ell$-tangent to the identity, and let $X$ be its associated $C^{\ell + r}$ vector field. We will show that $X$ is $C^r$ conjugate to $X_{\ell}(x):=x^{\ell+1}$. This allows concluding the proof. Indeed, if $g$ is another expanding germ in $\mathrm{Diff}^{\ell + 1 + r}_+(\mathbb{R},0)$ that is exactly $\ell$-tangent to the identity, then its associated vector field $Y$ will also be $C^r$ conjugate to $X_{\ell}$. Therefore, $X$ and $Y$ will be $C^r$ conjugate, and this will allow conjugating their time-1 maps, that is, $f$ and $g$.

Now, by Lemma~\ref{l:reduc}, we may assume that $X(x)=x^{\ell+1}(1+\delta(x))$ with $\delta(x)=o(x^{r-1})$. Moreover, by Lemma \ref{l:taylor} (applied to $k = \ell + r$ and $m = \ell$), we may also assume that $\delta$ is of class $C^{r-1}$. Equivalently, we can write $X(x)= \frac{x^{\ell +1}}{1+\epsilon(x)}$,  still with $\epsilon$ of class $C^{r-1}$ and $(r-1)$-flat. To prove that $X$ is $C^{r}$ conjugate to $X_\ell$, it suffices to prove that the solutions $h$ of $Dh=X_\ell \circ h / X$ on~$\R_+^*$ (near $0$), which are of class $C^{\ell+r+1}$, extend to $C^{r}$ diffeomorphisms at the origin. Assume without loss of generality that $X$ is expanding on an interval containing the origin and $1$, and consider the solution $h$ fixing $1$. We have:
$$\int_1^{h(x)}\frac{1}{X_{\ell}} = \int_1^{x}\frac{Dh}{X_\ell \circ h}=\int_1^x \frac1X,$$
that is 
$$-\frac1{\ell \, (h(x))^\ell} = -\frac1{\ell \, x^\ell} + \int_1^x\frac{\epsilon(y)}{y^{\ell+1}}dy+c$$
for some constant $c$. Equivalently,
$$\frac1{(h(x))^\ell} = \frac1{x^\ell} - \ell \, u(x), 
\quad \mbox{ with }\quad u(x)=\int_1^x\frac{\epsilon(y)}{y^{\ell +1}}dy+c.$$
This implies
$$
{(h(x))^\ell} = \frac1{\frac1{x^\ell}(1-\ell \, x^\ell \, u(x))} =  \frac{x^\ell}{1-\ell \, x^\ell \, u(x)},
$$
so that 
$$h(x) =\frac{x}{(1-\ell \, x^\ell \, u(x))^{1/\ell}}.$$
Recall that $\epsilon$ is $C^{r-1}$, so $u$ is $C^{r}$ outside $0$. Moreover, $\epsilon$ is $(r-1)$-flat, so $\frac{\epsilon(y)}{y^{\ell +1}} =o(y^{r-\ell -2})$, and therefore, since $r \leq \ell$,  
\begin{equation}\label{eq:n=0}
u(x)=o(x^{r-\ell -1}).
\end{equation} 
In particular, $x^\ell u(x)\to 0$ as $x \to 0$, since $r\ge1$. 
According to Lemma~\ref{l:debile}, it thus suffices to prove that $x\mapsto x(1-\ell \, x^\ell \, u(x))$, or equivalently $v:x\mapsto x^{\ell+1} \, u(x)$, is~$C^{r}$. To do this, observe that 
$$Dv(x) = (\ell +1)x^{\ell}u(x)+x^{\ell +1} Du(x) = (\ell +1)x^{\ell}u(x)+\epsilon(x).$$
We already know that $\epsilon$ is $C^{r-1}$, so we need only check that $w:x\mapsto x^\ell u(x)$ is $C^{r-1}$. On $\R_+^*$, $w$ is $C^{r}$ (as $u$) and satisfies
$$D^{(r-1)}w (x) = \sum_{n=0}^{r-1}a_n x^{\ell-(r-1-n)} D^{(n)}u (x) = \sum_{n=0}^{r}a_n x^{\ell -r+1+n} D^{(n)}u (x)$$
for some constants $a_n$. In view of this, it suffices to check that $x^{n+\ell +1-r} D^{(n)} u (x)$ has a limit at $0$ for every $n$ between $0$ and $r-1$. We have already checked this for $n=0$ (see (\ref{eq:n=0}) above). For $n\ge1$, 
\begin{eqnarray*}
x^{n+\ell +1-r} D^{(n)} u (x) 
&=& x^{n+\ell +1-r} \, D^{(n-1)} \! \left( \frac{\epsilon(x)}{x^{\ell +1}} \right) (x)\\
&=& x^{n+\ell +1-r}\sum_{j=0}^{n-1} c_{n,j} \, x^{-\ell-1-(n-1-j)} D^{(j)} \epsilon (x)\\
&=& \sum_{j=0}^{n-1}c_{n,j} \, x^{j-r+1} D^{(j)} \epsilon (x)
\end{eqnarray*}
for certain constants $c_{n,j}$. Finally, since $\epsilon$ is $(r-1)$-flat, we have $D^{(j)} \epsilon (x)=o(x^{r-1-j})$ for $j\le r-1$, which allows showing that $x^{n+\ell +1-r} D^{(n)} u (x) \to 0$ as $x \to 0$, thus completing the proof.
\end{proof}


\section{Conjugacies in case of coincidence of residues}
The goal of this section is to prove Theorem \ref{t:C}. Again, passing to the associated vector fields, this will follow from the next proposition:

\begin{prop}
\label{p:takens-fini}
If $\ell \geq 1$ and $r\ge \ell +2$, then every germ of  $C^{\ell+r}$ expanding vector field that is exactly $\ell$-flat is $C^{r}$ conjugate to a (unique) vector field of the form $x\mapsto x^{\ell+1}+\mu x^{2\ell +1}$. This still holds for $r=\ell+1$ if one further assumes that the vector field has $(2\ell+2)$ derivatives at~$0$ (and is not only of class $C^{2\ell+1}$).
\end{prop}

\begin{rem}
In the case $r=\ell+1$, the hypothesis of existence of $(2\ell+2)$ derivatives at the origin is necessary. Indeed, in \S \ref{section-ejemplo-helene}, there is an example of a $1$-flat $C^3$ vector field that is not $C^2$ conjugate to its normal form; this provides a counterexample for $\ell=1$ and $r=2=\ell+1$.
\end{rem}

\begin{proof}[Proof of Theorem \ref{t:C} from Proposition \ref{p:takens-fini}] 
Let $f$ and $g$ be two expanding germs as in the theorem, and let $X$ and $Y$ be their associated vector fields. By Proposition \ref{p:takens-fini}, these are $C^r$ conjugate to some $X_0$ and $Y_0$ of the form $x\mapsto x^{\ell+1}+\mu x^{2\ell+1}$ and $x\mapsto x^{\ell+1}+\nu x^{2\ell+1}$, respectively. Now, according to \S \ref{sub-residues-fields} and the hypothesis of coincidence of residues, 
$$\mu=-\mathrm{Resit}(f)=-\mathrm{Resit}(g)=\nu.$$ 
Thus, $X_0 = Y_0$ and, therefore, $X$ and $Y$ are $C^r$ conjugate, and so $f$ and $g$ as well.
\end{proof}

\vspace{0.1cm}

To prove Proposition \ref{p:takens-fini}, we first need a version of Lemma \ref{l:reduc} for large derivatives.

\vspace{0.1cm}

\begin{lem}\label{l:reduc2}
Let $\ell \geq 1$ and $r \geq 1$, and let $Y$ be a germ of $C^{\ell + r}$ vector field admitting a Taylor series expansion of order $\ell + r$ at the origin that starts with $\alpha x^{\ell+1}$, where $\alpha > 0$. If $r \ge \ell+1$, then $Y$ is smoothly conjugate to a vector field of the form $x\mapsto x^{\ell+1}+\mu x^{2\ell+1}+o(x^{\ell + r})$.
\end{lem}

\begin{proof}
Thanks to Lemma \ref{l:reduc}, we can start with $Y$ of the form $Y(x) = x^{\ell +1}+ o(x^{2\ell})$, which can be rewritten in the form $Y(x) = x\mapsto x^{\ell+1}+\mu x^{2\ell+1}+ o(x^{2\ell+1})$ since $Y$ is of class $C^{2\ell+1}$. One then constructs conjugate vector fields $Y_\ell :=Y, \ldots , Y_{r-1}$ such that each $Y_s$ is of the form $x\mapsto x^{\ell+1}+\mu x^{2\ell+1}+ o(x^{\ell +s+1})$. The construction is made by induction by letting $Y_{s+1}=h^* Y_s$, with $h(x)=x+ax^{\ell +s+2}$ for some well-chosen $a$. Details are left to the reader. 
\end{proof}

\vspace{0.1cm}

\begin{proof}[Proof of Proposition \ref{p:takens-fini}] 
Thanks to Lemma \ref{l:reduc2}, we can start with a vector field $X$ globally defined on $\mathbb{R}_+$ that has the form 
$$X(x) = x^{\ell+1}+\mu x^{2\ell+1}+o(x^{\ell + r}).$$
We then define the function $\epsilon$ by the equality
$$X(x) = x^{\ell +1}+\mu x^{2\ell +1}+x^{2\ell + 1}\epsilon(x).$$ 
We want to prove that $X$ is $C^{r}$ conjugate to $X_0 (x) := x^{\ell+1}+\mu x^{2\ell + 1}$. For $y\in[0,1]$, let $X_y (x) := X_0(x)+yx^{2\ell + 1}\epsilon(x)$, so that $X_1=X$. Let $Y$ be the horizontal $C^{\ell + r}$ vector field on $S := \R_+\times[0,1]$ defined by $Y(x,y)=X_y(x) \, \partial_x + 0 \cdot \partial_y$. \medskip

We claim that it suffices to find a $C^{r}$ vector field $Z$ on $S$ of the form $(x,y)\mapsto K(x,y) \, \partial_x + \partial_y$, with $K$ vanishing on $\{0\}\times[0,1]$ and $[Z,Y]=0$ near $\{0\}\times[0,1]$. To show this, denote by $\phi_Z^t$ the flow of such a $Z$. Clearly, $\phi_Z^1 |_{|\R_+\times \{0\}}$ is of the form $(x,0)\mapsto (\varphi(x),1)$ for some $C^{r}$ diffeomorphism $\varphi$ of $\R_+$. We claim that $\varphi_* X_0 = X_1$. Indeed, the equality $[Z,Y]=0$ implies that the flows of $Z$ and $Y$ commute near the ``vertical'' segment $\{0\}\times [0,1]$. Let $\mathcal{U}$ be a set of the form $\{\phi_Z^t(x,0): (x,t)\in [0,\eta] \times [0,1]\}$ where these flows commute. Then, if $f_0^1$ and $f_1^1$ denote the time-$1$ maps of $X_0$ and $X_1$, respectively, for every $x\in f_0^{-1}([0,\eta])$, 
$$(\varphi(f_0(x)),1)=\phi_Z^1(f_0(x),0) = \phi_Z^1(\phi_Y^1(x,0)) = \phi_Y^1(\phi_Z^1(x,0))= \phi_Y^1(\varphi(x),1) = (f_1(\varphi(x)),1).$$
Therefore, $\varphi$ conjugates $f_0$ to $f_1$, and thus $X_0$ to $X_1$ (by the uniqueness of the vector fields associated to diffeomorphisms), as required. \medskip

We are thus reduced to proving the existence of a $Z$ as above. Let us leave aside the questions of regularity for a while and assume that all involved functions are $C^\infty$. In this case, according for example to \cite[Proposition 2.3]{takens}, the function $\epsilon$ can be assumed to be of the form $x\mapsto x \, \delta(x)$, with $\delta$ smooth at $0$. We look for $Z$ of the form 
$$Z(x,y)=x^{\ell + 1}H(x,y)\partial_x+\partial_y.$$
 Let $F(x) := 1+\mu x^\ell$, so that $X_0(x)=x^{\ell + 1}F(x)$. If $(Y_x,Y_y)$ and $(Z_x,Z_y)$ denote the coordinates of $Y$ and $Z$, respectively, then we have $[Z,Y] = \psi(x,y)\partial_x+0\cdot\partial_y$, with
\begin{eqnarray*}
\psi(x,y) 
&=& \frac{\partial Y_x}{\partial x}Z_x - \frac{\partial Z_x}{\partial x}Y_x+\frac{\partial Y_x}{\partial y}Z_y \\
&=& \left[(\ell+1)x^\ell F(x) + (2\ell + 1)x^{2\ell}y\epsilon (x)+yx^{2\ell +1} D\epsilon(x)\right]\times x^{\ell +1}H(x,y) \\
&\qquad \,\, -& \left((\ell+1)x^{\ell}H(x,y)+x^{\ell +1}\frac{\partial H}{\partial x}(x,y)\right)\times\left(x^{\ell + 1}F(x)+x^{2\ell + 1}y\epsilon (x)\right) 
\, + \,  x^{2\ell + 1}\epsilon (x)\\
&=&  x^{2\ell + 2}\left(\frac{\partial H}{\partial x}(x,y)(F(x)+x^\ell y\epsilon (x))
+H(x,y)y\left( \ell x^{\ell -1}\epsilon (x) + x^\ell D \epsilon(x)\right)+\delta(x)\right)
\end{eqnarray*}
(the terms involving $x^{2\ell +1}$ cancel). We hence need to solve an equation of the form 
\begin{equation}
\label{e:EDL}
a(x,y) \frac{\partial H}{\partial x}(x,y)+b(x,y)H(x,y)+\delta(x)=0,
\end{equation}
with 
$$a(x,y) := F(x)+x^\ell y\epsilon (x), 
\qquad b(x,y) := y \, D(x^{\ell} \epsilon(x)), 
\qquad \delta (x) := \frac{\epsilon(x)}{x}.$$
This is a family of linear differential equations with variable $x$ and parameter $y$, with $a$ nonvanishing near $\{0\}\times [0,1]$ and $a$, $b$ and $\delta$ smooth (for now). The resolution of such equations ensures the existence of a smooth solution $H$ (which is not unique, because of the choice of integration constants), which concludes the proof in the smooth case. \medskip

Let us now concentrate on our case, where $\ell \geq 1$, $r\ge \ell+1$ and $X$ is only $C^{\ell+r}$ (with $(2\ell + 2)$ derivatives at $0$ if $r = \ell + 1$). We want to check that, in this case, there exists $H$ satisfying (\ref{e:EDL}) such that $(x,y)\mapsto x^{\ell+1}H(x,y)$ is $C^{r}$. 

\medskip

The existence of an $H$ satisfying (\ref{e:EDL}) that is (smooth in $y$ and) $C^1$ in $x$ follows from that $a,b$ and $\delta$ above are (smooth in $y$ and) continuous in $x$, as we now check. Since we already know that $\epsilon$ is as regular as $X$ away from $0$, that is, $C^{\ell + r}$, we are left with checking that $x\mapsto x^\ell\epsilon(x)$ and $x\mapsto \frac{\epsilon(x)}x$ extend respectively to a $C^1$ and a $C^0$ function near $0$. The first point follows from Lemma \ref{l:taylor} applied to $u(x):=x^{2\ell+1}\epsilon(x)$, $k=\ell+r$ and $m=\ell$ (which actually shows that $x\mapsto x^\ell\epsilon(x)$ is $C^{r-1}$ with $r-1\ge \ell\ge 1$). For the second point, if $r\ge \ell+2$, we know that
$$x^{2\ell + 1} \epsilon(x) = o(x^{\ell+r})$$
and $\ell+r\ge 2\ell+2$, hence by dividing each term by $x^{2\ell + 1}$ we get $\epsilon(x)=o(x)$, so $\delta$ indeed extends as a continuous function at $0$. If $r=\ell+1$, this is precisely where we use the additional assumption that $X$ is $(2\ell+2)$ times differentiable at $0$, which by Lemma \ref{l:reduc2} implies that one can start with $X$ of the form 
$$X(x) = x^{\ell+1}+\mu x^{2\ell+1}+o(x^{2\ell+2}),$$
hence  $x^{2\ell+1}\epsilon(x)=o(x^{2\ell+2})$ and thus again $\epsilon(x)=o(x)$. \medskip

To conclude the proof, we finally need to show that $(x,y)\mapsto x^{\ell+1} H(x,y)$ is $C^r$. To do this, we will actually prove by induction on $s \in [\![0,\ell + 1]\!]$ that $H_s \!: (x,y)\mapsto x^sH(x,y)$ is $C^{r-\ell+s-1}$. \medskip

\noindent \emph{Case $s=0$}. We have already given part of the ingredients for this initial case. If $r=\ell+1$, then $r-\ell+0-1=0$, and we already know that $H_0=H$ is continuous. 

Now consider $r\ge \ell+2$. The theory of linear differential equations tells us that $H=H_0$ has one degree of differentiability more than the coefficients of the equation, so it suffices to check that these are $C^{r-\ell-2}$. We have already seen that $a$ and $b$ in \eqref{e:EDL} were respectively of class $C^{r-1}$ and $C^{r-2}$, so \emph{a fortiori} $C^{r-\ell-1}$ since $\ell\ge 1$. Now applying Lemma \ref{l:taylor} to $u(x):=x^{2\ell+1}\epsilon(x)$, $k=\ell+r$ and $m=2\ell$, we get that $\epsilon$ is $C^{r-\ell-1}$. Moreover, since $x^{2\ell + 1} \epsilon(x) = o(x^{\ell+r})$, we have $\epsilon(x)=o(x^{r-\ell-1})$, so $\epsilon$ is $(r-\ell-1)$-flat. Applying Lemma \ref{l:taylor} again but this time to $u:=\epsilon$, $k=r-\ell-1\ge 1$ and $m=0$, we get that $\delta:x\mapsto \frac{\epsilon(x) }x$ is $C^{r-\ell-2}$, as required.\medskip

\noindent \emph{Inductive step}. Assume now that the induction hypothesis is true for some $0 \le s \le \ell$. Then 
$$\frac{\partial H_{s+1}}{\partial x}(x,y) = (s+1)x^s H(x,y) + x^{s+1}\frac{\partial H}{\partial x}(x,y).$$
The first term of the right hand side is $C^{r-\ell+s-1}$ by the induction hypothesis, and the second is equal to
$$
x^{s+1}\left(-\frac{b(x,y)}{a(x,y)}H(x,y) -\frac{c(x,y)}{a(x,y)}\right)
= -\frac{x\cdot b(x,y)}{a(x,y)}H_s (x,y) -\frac{x^{s+1}c(x,y)}{a(x,y)}.
$$
Since $a$ is $C^{r-1}$ and thus $C^{r-\ell+s-1}$, we are left with proving that the numerators of the fractions, namely
$$y (\ell x^\ell\epsilon (x)+x^{\ell+1} D\epsilon(x)) \qquad \text{and}\qquad x^s \epsilon (x),$$ 
are $C^{r-\ell+s-1}$. For the second one, this follows from Lemma \ref{l:taylor} still applied to $u(x):=x^{2\ell+1}\epsilon(x)$ and $k=\ell+r$ but this time with $m=2\ell-s\le 2\ell<k$. A last application with $m=r-2<k$ shows that $x\mapsto x^{\ell+1}\epsilon(x)$ is $C^{\ell+r-m-1}=C^{r+1}$, so that the first numerator is $C^r$ and \emph{a fortiori} $C^{r-\ell+s-1}$ since $s\le \ell$. This concludes the induction and thus the proof.
\end{proof}

\begin{rem}
\label{r:final}
Proposition \ref{p:takens-fini} implies Takens' normal form result: if $X$ is a $C^{\infty}$ expanding vector field that is exactly $\ell$-flat for some $\ell \geq 1$, then $X$ is $C^{\infty}$ conjugate to a vector field of the form $x \mapsto x^{\ell+1} + \mu \, x^{2\ell +1}$ for a unique $\mu$. Indeed, we have shown the existence of such a $C^r$ conjugacy $\varphi_r$ for each $r \geq \ell + 2$. However, the conjugacy is unique up to composition with a member of the flow of $X$, which is a $C^{\infty}$ diffeomorphism. Therefore, $\varphi_r$ and $\varphi_{r+1}$ differ by the composition of a $C^{\infty}$ diffeomorphism, which easily implies that all the $\varphi_r$ are actually $C^{\infty}$. 

An analogous result holds for diffeomorphisms: any two $C^{\infty}$ germs that are exactly $\ell$-flat, both expanding or both contracting, and have the same iterative residue, are $C^{\infty}$ conjugate. Rather surprisingly, Takens' proof of this fact passes through the famous Borel Lemma on the realization of sequences of numbers as derivatives at the origin of $C^{\infty}$ germs, while our argument above avoids this.
\end{rem}


\section{Residues and distortion}
\label{section-last-question}

Recall that an element $g$ of a finitely-generated group is said to be {\em distorted} if $g^n$ may be written 
as a product of factors among the generators and their inverses whose total number grows sublinearly with respect 
to $n$. (This definition does not depend on the chosen generating system.) An element of a general group is a 
{\em distortion element} if it is distorted inside some finitely-generated subgroup. The next question is inspired from \cite{navas}:

\begin{qs} 
What are the distortion elements of the group $\mathrm{Diff}^{\omega}_+ (\mathbb{R},0)$ 
of germs of (orientation-preserving) real-analytic diffeomorphisms fixing the origin~?
\end{qs}

An example of distortion element in the group of germs above is 
$$g (x) := \frac{x}{1+x} = x - x^2 + x^3 - x^4 +\ldots$$
Indeed, letting $h (x) = \frac{x}{2}$, one has 
\begin{equation}\label{eq:affine-flow}
h \, g \, h^{-1} = g^2,
\end{equation}
which easily yields $g^{2^m} = h^m g h^{-m}$ for all $m \geq 1$. Using this, it is not hard to conclude that $g^n$ may be written as a product of $O(\log(n))$ factors $g^{\pm 1}, h^{\pm 1}$. The case of the map below is particularly challenging.

\begin{qs} 
Is the germ $f(x) := x - x^2$ a distortion element of $\mathrm{Diff}^{\omega}_+ (\mathbb{R},0)$~?
\end{qs}

Observe that  $f$  is a distortion element of the much larger group $\mathrm{Diff}^1_+(\mathbb{R},0)$. Indeed, by Theorem \ref{t:B}, it is $C^1$ conjugated to $g$, which is a distortion element. Actually, this is a particular case of the much more general result below.

\begin{prop} 
Every (nontrivial) parabolic germ of real-analytic diffeomorphism of the line fixing the origin is a distortion element of the group $\mathrm{Diff}^1_+(\mathbb{R},0)$.
\end{prop}

\begin{proof}  
Let $f$ be such a germ and $\ell$ its order of contact with the identity.  By Proposition \ref{p:Takens}, 
$f$ is $C^1$ conjugated to a germ $g$ of the form
$$g(x) := \frac{x}{\sqrt[\ell]{1 \pm x^{\ell}}} = x \mp \frac{x^{\ell +1}}{\ell} + \ldots .$$
We are hence left to showing that $g$ is a distortion element. But as above, this follows from the relation 
\, $h_\ell \, g \, h_\ell^{-1} = g^2$, \, where $h_\ell (x) := \frac{x}{2^{{1/\ell}}}$.
\end{proof}

We do not know whether the germ $f$ above is a distortion element of the smaller group $\mathrm{Diff}^2_+(\mathbb{R},0)$. A first difficulty of this question lies in that one cannot deduce distortion from a relation of type (\ref{eq:affine-flow}) because of the nonvanishing of the iterative residue of $f$, as proven below.

\begin{prop} 
Let $g$ be a germ of $C^3$ diffeomorphism fixing the origin that is exactly $1$-tangent to the identity. If $g$ is $C^2$ conjugate to some element $g^t \in \mathrm{Diff}_+^3 (\mathbb{R},0)$ of its flow, 
with $t \neq 1$, then $\mathrm{Resit}(g)$ vanishes.
\end{prop}

\begin{proof} 
Since $\mathrm{Resit} (g)$ is invariant under $C^2$ conjugacy, by (\ref{eq:resit-flow}) 
(see also Remark (\ref{rem:resit-flow})), one has
$$\mathrm{Resit} (g) = \mathrm{Resit} (g^t) = \frac{\mathrm{Resit}(g)}{|t|}.$$
This easily implies $\mathrm{Resit} (g) = 0$ if $t \neq 1$. The case $t = -1$ is impossible since 
a contracting diffeomorphisms cannot be conjugated to a expanding one.
\end{proof}

To finish, let us mention that we do not even know whether the germ $f$ above is a distortion element of the group $\widehat{\mathrm{Diff}} (\mathbb{R},0)$ of formal germs of diffeomorphisms. More generally, the next question seems challenging:

\begin{qs} 
Given a field $\mathbb{K}$, what are the distortion elements of $\ger$~? 
What about $\widehat{\mathrm{Diff}} (\mathbb{Z},0)$~?
\end{qs}

\vspace{0.3cm}

\noindent{\bf Acknowledgments.} 
H\'el\`ene Eynard-Bontemps was funded by the IRGA project ADMIN of UGA and the CNRS in the context of a d\'el\'egation. She would like to thank CMM (FB210005, BASAL funds for centers of excellence from ANID-Chile)\,/\,U. de Chile for the hospitality during this 
period and Michele Triestino for inspiring discussions related to this work. Andr\'es Navas was funded by Fondecyt Research Project 1220032, and would like to thank Adolfo Guillot, Jan Kiwi and Mario Ponce for useful discussions and insight concerning residues of parabolic germs in the complex setting.


\vspace{0.5cm}

\begin{small}

\end{small}

\begin{footnotesize}

\vspace{0.42cm}

\noindent {\bf H\'el\`ene Eynard-Bontemps} \hfill{\bf Andr\'es Navas}

\noindent Universit\'e Grenoble Alpes, CNRS\hfill{ Dpto. de Matem\'atica y C.C.}

\noindent  Institut Fourier  \hfill{ Universidad de Santiago de  Chile}

\noindent 38000 Grenoble \hfill{Alameda Bernardo O'Higgins 3363}

\noindent France \hfill{Estaci\'on Central, Santiago, Chile} 

\noindent helene.eynard-bontemps@univ-grenoble-alpes.fr \hfill{andres.navas@usach.cl}

\noindent \& 

\noindent Center for Mathematical Modeling 

\noindent FCFM, Universidad de Chile 

\noindent Santiago, Chile 

\end{footnotesize}

\end{document}